\documentclass[12pt, reqno]{amsart}
\usepackage{amsmath, amsthm, amscd, amsfonts, amssymb, graphicx, color}
\usepackage[bookmarksnumbered, colorlinks, plainpages]{hyperref}

\usepackage[all]{xy}
\everymath{\displaystyle}

\title{Scaled-free objects II}

\author[W.\ Grilliette]{Will Grilliette}

\address{Department of Mathematics, Texas State University, 601 University Dr., San Marcos, TX  78666}
\email{w.b.grilliette@gmail.com}



\date{Received: xxxxxx; Revised: yyyyyy; Accepted: zzzzzz.}

\newtheorem{thm}{Theorem}[subsection]
\newtheorem{prop}[thm]{Proposition}
\newtheorem{cor}[thm]{Corollary}
\newtheorem{lem}[thm]{Lemma}

\newtheorem{thm2}{Theorem}[section]
\newtheorem{prop2}[thm2]{Proposition}
\newtheorem{lem2}[thm2]{Lemma}
\newtheorem{cor2}[thm2]{Corollary}

\theoremstyle{definition}
\newtheorem{defn}[thm]{Definition}
\newtheorem{defn2}[thm2]{Definition}

\theoremstyle{remark}
\newtheorem{ex}[thm]{Example}
\newtheorem{ex2}[thm2]{Example}

\newcommand{\alg}[1]{\mathcal{#1}}

\newcommand{\Set}{\mathbf{Set}}
\newcommand{\WSetC}{\mathbf{WSet}_{1}}
\newcommand{\WSetB}{\mathbf{WSet}_{\infty}}
\newcommand{\AWSetC}{\mathbf{AWSet}_{1}}
\newcommand{\AWSetB}{\mathbf{AWSet}_{\infty}}
\newcommand{\BanC}{\mathbf{Ban}_{1}}
\newcommand{\BanB}{\mathbf{Ban}_{\infty}}
\newcommand{\BanAlgC}{\mathbf{BanAlg}_{1}}
\newcommand{\BanAlgB}{\mathbf{BanAlg}_{\infty}}
\newcommand{\MBanC}{\mathbf{MBan}_{1}}
\newcommand{\MBanB}{\mathbf{MBan}_{\infty}}
\newcommand{\MBanAlgC}{\mathbf{MBanAlg}_{1}}
\newcommand{\MBanAlgB}{\mathbf{MBanAlg}_{\infty}}

\newcommand{\B}{\mathcal{B}}
\newcommand{\CB}{\mathcal{CB}}
\newcommand{\Th}{\mathcal{T}_{h}}

\DeclareMathOperator{\Ran}{ran}
\DeclareMathOperator{\Span}{span}

\DeclareMathOperator{\bnd}{bnd}
\DeclareMathOperator{\cbnd}{cbnd}
\DeclareMathOperator{\brn}{brn}

\DeclareMathOperator{\C}{C}
\DeclareMathOperator{\MC}{MC}
\DeclareMathOperator{\BanSp}{BanSp}
\DeclareMathOperator{\BanAlg}{BanAlg}
\DeclareMathOperator{\MBanSp}{MBanSp}
\DeclareMathOperator{\MBanAlg}{MBanAlg}
\DeclareMathOperator{\mA}{mA}
\DeclareMathOperator{\MA}{MA}
\DeclareMathOperator{\MIN}{MIN}
\DeclareMathOperator{\MAX}{MAX}
\DeclareMathOperator{\AMAX}{AMAX}

\begin{document}
\setcounter{page}{1}

\begin{abstract}
This work creates two categories of ``array-weighted sets'' for the purposes of constructing universal matrix-normed spaces and algebras.  These universal objects have the analogous universal property to the free vector space, lifting maps completely bounded on a generation set to a completely bounded linear map of the matrix-normed space.

Moreover, the universal matrix-normed algebra is used to prove the existence of a free product for matrix-normed algebras using algebraic methods.
\end{abstract}

\maketitle

\section{Introduction}

In \cite{grilliette1}, universal Banach spaces and algebras were constructed as left adjoints to forgetful functors to categories of weighted sets.  This paper considers a similar construction to build universal matricial Banach spaces and algebras from ``array-weighted'' sets.

Much like a weighted set is a Banach space stripped of its vector space structure, an ``array-weighted set'' will be a matricial Banach space stripped of its vector space structure, leaving a set with a net of nonnegative-valued functions.  Thus, the categories of array-weighted sets are a proposed replacement to the category of sets for construction of matrix-normed objects.  As such, one can then consider adapting algebraic constructions to matricial Banach algebras, such as generators and relations from \cite{blackadar1985,gerbracht,goodearl,loring2009,nassopoulos2008,pelletier,cat-top}.

Section \ref{preliminaries} establishes notation and existing results, which will be used throughout the paper for weighted sets and matrix-normed spaces.  Section \ref{awsets} develops the categories of array-weighted sets.  Section \ref{constructions} then builds the matrix-normed space for a given array-weighted set, showing several key examples to highlight the resulting structure.  Finally, Section \ref{mbanachalgebra} describes a theory of matricial Banach algebras generalizing the theory of operator algebras.  In particular, Section \ref{free-product} uses the universal matricial Banach algebra to show the existence of the free product of matricial Banach algebras using an algebraic proof.

The author would like to thank the referees of this paper for their comments and patience in its revision.  The author would also like to extend his thanks to Dr.\ Nathan Smith from the University of Texas at Tyler for the conversation which developed Example \ref{badness4}.

\section{Preliminaries}\label{preliminaries}

This section covers some previous results which are either motivating for the current work or needed for the current work's constructions.  In particular, the primary results regarding weighted sets and their constructions are reviewed from \cite{grilliette1} in Section \ref{weightedsets}.  Likewise, some foundational results into matrix-normed spaces are revisited from \cite{blecher1991,effros1988,paulsen,ruan1988} in Section \ref{matrix-normed} and \ref{matrix-normed-constructions}.  However, be aware that while \cite{grilliette1} considered Banach spaces for both real and complex fields, all vector spaces and algebras for the current work will be over $\mathbb{C}$ to be consistent with the literature of matrix-normed spaces.  Moreover, this section sets the notation that will be used throughout the current work.

\subsection{Weighted sets and their constructions}\label{weightedsets}

In \cite{grilliette1}, universal Banach spaces and algebras were constructed as left adjoints of forgetful functors to categories of ``weighted sets''.  The object ``weighted set'' has gone by different names in previous works:  ``bewertete Menge'' in \cite[Definition 1.1.7]{gerbracht}, ``crutched set'' in \cite[p.\ 14]{grilliette0}, and ``normed set'' in \cite[p.\ 7]{grandis2004}, where $\infty$ is allowed as a weight value.  This paper will use the following terminology, conventions, and categories.

\begin{defn}[{Weighted set conventions, \cite{grilliette1}}]
A \emph{weighted set} is a set $S$ equipped with a weight function $w_{S}:S\to[0,\infty)$.  Given two weighted sets $S$ and $T$, a function $\phi:S\to T$ is \emph{bounded} if there is $L\geq 0$ such that for all $s\in S$, $w_{T}\left(\phi(s)\right)\leq L\cdot w_{S}(s)$.  Let
\[
\bnd(\phi):=\inf\left\{L\in[0,\infty):w_{T}\left(\phi(s)\right)\leq L\cdot w_{S}(s)\forall s\in S\right\},
\]
the \emph{bound constant} of $\phi$.  If $\bnd(\phi)\leq 1$, $\phi$ is \emph{contractive}.  Let $\WSetC$ denote the category of weighted sets with contractive maps, and $\WSetB$ denote the category of weighted sets with bounded maps.
\end{defn}

A Banach space stripped of its vector space structure is a weighted set, and that relationship can be encapsulated in a forgetful functor.  The main content of \cite[Theorem 3.1.1]{grilliette1} is that this forgetful functor has a left adjoint, constructing a Banach space from a weighted set.

\begin{defn}[{Scaled-free Banach space, \cite[p.\ 281]{grilliette1}}]
For a weighted set $S$, let $\hat{S}:=S\setminus w_{S}^{-1}(0)$, all elements with nonzero weight.  Define the discrete measure $\mu_{S}:\mathcal{P}\left(\hat{S}\right)\to[0,\infty]$ by $\mu_{S}(T):=\sum_{s\in T}w_{S}(s)$.  The \emph{scaled-free Banach space} of $S$ is
\[
\BanSp(S):=\ell^{1}\left(\hat{S},\mu_{S}\right),
\]
a weighted $\ell^{1}$-space over $\mathbb{C}$.  Define $\zeta_{S}:S\to\BanSp(S)$ by
\[
\zeta_{S}(s):=\left\{\begin{array}{cc}
0,	&	s\not\in\hat{S},\\
\delta_{s},	&	s\in\hat{S},\\
\end{array}\right.
\]
where $\delta_{s}$ is the point mass at $s\in\hat{S}$.
\end{defn}

\begin{thm}[{Universal property of the scaled-free Banach space, \cite[Theorem 3.1.1]{grilliette1}}]
Let $\BanB$ be the category of Banach spaces with bounded linear maps, and $F_{\BanB}^{\WSetB}:\BanB\to\WSetB$ the forgetful functor stripping all linear structure.  For a weighted set $S$ and a Banach space $W$, consider a bounded function $\phi:S\to F_{\BanB}^{\WSetB}(W)$.  Then, there is a unique bounded linear map $\hat{\phi}:\BanSp(S)\to W$ such that $F_{\BanB}^{\WSetB}\left(\hat{\phi}\right)\circ\zeta_{S}=\phi$.  Moreover,
\[
\bnd(\phi)=\left\|\hat{\phi}\right\|_{\B\left(\BanSp(S),W\right)}.
\]
\end{thm}

Likewise, one would like to construct a Banach algebra from a weighted set in a similar fashion.  While the functor $\BanSp$ creates a linear structure, a multiplicative structure can be created using the construction of the Banach tensor algebra.

\begin{defn}[{Banach tensor algebra, \cite[p.\ 165]{leptin1969}}]
Let $\BanC$ be the category of Banach spaces with contractive linear maps.  For Banach space $V$, inductively define the \emph{projective tensor powers} of $V$ in the following way:
\[\begin{array}{ccc}
V^{\hat{\otimes}1}:=V,	&	V^{\hat{\otimes}(n+1)}:=\left(V^{\hat{\otimes}n}\right)\hat{\otimes}V	&	\forall n\in\mathbb{N},
\end{array}\]
where $\hat{\otimes}$ denotes the projective tensor product.  The \emph{Banach tensor algebra} of $V$ is
\[
\mathcal{T}(V):={\coprod_{n\in\mathbb{N}}}^{\BanC}V^{\hat{\otimes} n},
\]
the $\ell^{1}$-direct sum of these projective tensor powers, equipped with the usual tensor multiplication determined by the canonical isomorphism $V^{\hat{\otimes}m}\hat{\otimes} V^{\hat{\otimes}n}\to V^{\hat{\otimes}(m+n)}$.  Define $\iota_{V}:V\to\mathcal{T}(V)$ to be the inclusion map into the first tensor power of $V$ in $\mathcal{T}(V)$
\end{defn}

\begin{thm}[{Universal property of the Banach tensor algebra, \cite[Satz 1]{leptin1969}}]
Let $\BanAlgC$ be the category of Banach algebras with contractive algebra homomorphims, and $F_{\BanAlgC}^{\BanC}:\BanAlgC\to\BanC$ the forgetful functor stripping multiplicative structure.  For a Banach space $V$ and a Banach algebra $\alg{B}$, consider a contractive linear map $\phi:V\to F_{\BanAlgC}^{\BanC}(\alg{B})$.  Then, there is a unique contractive algebra homomorphism $\hat{\phi}:\mathcal{T}(V)\to\alg{B}$ such that $F_{\BanAlgC}^{\BanC}\left(\hat{\phi}\right)\circ\iota_{V}=\phi$.
\end{thm}

Composing the left adjoints $\BanSp$ and $\mathcal{T}$ creates a new left adjoint, $\BanAlg:=\mathcal{T}\circ\BanSp$, with the following universal property.

\begin{thm}[{Universal property of the scaled-free Banach algebra, \cite[Theorem 3.2.4]{grilliette1}}]
Let $F_{\BanAlgC}^{\WSetC}:\BanAlgC\to\WSetC$ be the forgetful functor stripping all algebraic structure.  For a weighted set $S$ and a Banach algebra $\alg{B}$, consider a contractive map $\phi:S\to F_{\BanAlgC}^{\WSetC}(\alg{B})$.  Then, there is a unique contractive algebra homomorphism $\hat{\phi}:\BanAlg(S)\to\alg{B}$ such that $F_{\BanAlgC}^{\WSetC}\left(\hat{\phi}\right)\circ\iota_{\BanSp(S)}\circ\zeta_{S}=\phi$.
\end{thm}

\subsection{Matricial Banach spaces and important examples}\label{matrix-normed}
The central goal of this paper is to adapt the constructions of the previous section to ``matricial Banach spaces''.  The core idea of this structure is a Banach space equipped with norms on the matrices over the space that have a boundedness condition with the action of the scalar matrices.

However, since this idea is to be abstracted in Section \ref{awsets}, the presentation here will be categorical and functorial to keep notation consistent between general sets and vector spaces.  Fundamentally, an $m\times n$-matrix of elements is a function from a cartesian product into the appropriate target set, which is described below.

\begin{defn}[The functor $M_{m,n}$]\label{matrix-notation}
For $n\in\mathbb{N}$, let $[n]:=\{1,\ldots,n\}$, the set of the first $n$ natural numbers.  Letting $\Set$ denote the category of sets, define
\[
M_{m,n}(-):=\Set([m]\times[n],-)
\]
for $m,n\in\mathbb{N}$, a covariant hom-functor from $\Set$ to itself.  For a set $S$, $M_{m,n}(S)$ is the set of all functions from $[m]\times[n]$ to $S$.  An element $A$ of $M_{m,n}(S)$ is an \emph{array}, or \emph{matrix}, with entries from $S$.  For $1\leq j\leq m$ and $1\leq k\leq n$, the \emph{$j,k$-entry} of $A$ will be denoted with function notation as $A(j,k)$.

Moreover,given sets $S$ and $T$, the action of $M_{m,n}$ on a function $\phi:S\to T$ gives a function $M_{m,n}(\phi):M_{m,n}(S)\to M_{m,n}(T)$ defined entrywise by
\[
M_{m,n}(\phi)(A)(j,k):=\phi(A(j,k))
\]
for $1\leq j\leq m$ and $1\leq k\leq n$, simply applying $\phi$ to all entries of $A$.  This action is precisely the \emph{ampliation} of maps found in \cite{blecher1991,effros1988,paulsen,ruan1988}.
\end{defn}

If $S$ already has existing algebraic structure, said structure can be extended to $M_{m,n}(S)$.  Below are the conventions taken for this paper for vector spaces.

\begin{defn}[Matrix conventions, vector spaces]
For a vector space $V$ and $m,n\in\mathbb{N}$, $M_{m,n}(V)$ is equipped with the usual pointwise addition and scalar multiplication.  The set of scalar matrices will be distinguished by $\mathbb{M}_{m,n}:=M_{m,n}(\mathbb{C})$.  For $k,l\in\mathbb{N}$, the actions $\mathbb{M}_{k,m}\times M_{m,n}(V)\to M_{k,n}(V)$ and $M_{m,n}(V)\times\mathbb{M}_{n,l}\to M_{m,l}(V)$ will be by matrix multiplication.
\end{defn}

At last, the definition of a matrix-normed space can be given.

\begin{defn}[{Matrix-normed spaces, \cite[p.\ 264]{blecher1991}}]
For $n\in\mathbb{N}$, equip $\mathbb{C}^{n}$ with the Euclidean norm.  For $m,n\in\mathbb{N}$, let $\mathbb{M}_{m,n}$ be equipped with the operator norm from $\mathbb{C}^{n}$ to $\mathbb{C}^{m}$.  For a vector space $V$, a \emph{matrix-norm} on $V$ is a net $\left(\|\cdot\|_{V,m,n}\right)_{m,n\in\mathbb{N}}$ such that
\begin{enumerate}
\item $\|\cdot\|_{V,m,n}$ is a norm on $M_{m,n}(V)$,
\item $\|ABC\|_{V,k,l}\leq\|A\|_{\mathbb{M}_{k,m}}\|B\|_{V,m,n}\|C\|_{\mathbb{M}_{n,l}}$
\end{enumerate}
for all $k,m,n,l\in\mathbb{N}$, $A\in\mathbb{M}_{k,m}$, $B\in M_{m,n}(V)$, $C\in\mathbb{M}_{n,l}$.  A vector space $V$ equipped with such a matrix-norm is a \emph{matrix-normed space}.
\end{defn}

Be aware that \cite[p.\ 246]{effros1988} and \cite[p.\ 1]{ruan1988} use an alternate set of axioms only involving square matrices.  However, by \cite[Exercises 13.1-2]{paulsen}, the two are interchangeable.

For maps between matrix-normed spaces, the linear maps between each level of matrices are required to be bounded by a uniform constant.  The standard definition given in \cite{blecher1991,effros1988,paulsen,ruan1988} uses square matrices.  Since this paper will be handling specifically nonsquare matrices in Section \ref{awsets}, an equivalent formulation will be used, which was referenced in \cite[p.\ 246]{effros1988}.

\begin{defn}[Completely bounded maps]
Given matrix-normed spaces $V$ and $W$, a linear map $\phi:V\to W$ is \emph{completely bounded} if
\begin{enumerate}
\item $M_{m,n}(\phi)$ is bounded for all $m,n\in\mathbb{N}$,
\item $\|\phi\|_{\CB(V,W)}:=\sup\left\{\left\|M_{m,n}(\phi)\right\|_{\B\left(M_{m,n}(V),M_{m,n}(W)\right)}:m,n\in\mathbb{N}\right\}<\infty$.
\end{enumerate}
The map $\phi$ is \emph{completely contractive} if $\|\phi\|_{\CB(V,W)}\leq 1$.
\end{defn}

Notably, a matrix-normed space $V$ is a normed space when stripped of all its matrix-norms, except for the norm on $M_{1,1}(V)\cong V$.  To compare matrix-normed spaces with Banach spaces, the current work will require that this underlying normed space be complete.  As noted in \cite[p.\ 246]{effros1988}, the underlying normed space is complete if and only if all the matrix levels above it are as well.  Hence, the following definitions are made unambiguously.

\begin{defn}[Matricial Banach space]
A complete matrix-normed space is a \emph{matricial Banach space}.  Let $\MBanB$ be the category of matricial Banach spaces with completely bounded linear maps, and $\MBanC$ be the category of matricial Banach spaces with completely contractive linear maps.
\end{defn}

For a Banach space, one would like to extend its existing norm to a matrix norm.  However, such an extension is not unique, as shown in the following standard constructions.

\begin{defn}[{Minimal operator space structure, \cite[Theorem 2.1]{effros1988}}]\label{MIN-definition}
Given a Banach space $V$, let $\MIN(V)$ be $V$ equipped with the matrix-norm given by
\[
\|A\|_{\MIN(V),m,n}:=\sup\left\{\left\|M_{m,n}(\phi)(A)\right\|_{\mathbb{M}_{m,n}}:\phi\in V^{*},\|\phi\|_{V^{*}}\leq 1\right\},
\]
the injective tensor norm on $V\otimes\mathbb{M}_{m,n}$.
\end{defn}

\begin{defn}[{Maximal operator space structure, \cite[Example 2.4]{blecher1991}}]
For a Hilbert space $H$ and $n\in\mathbb{N}$, let $H^{(n)}$ denote the $\ell^{2}$-direct sum of $H$ with itself $n$ times.  Given a Banach space $V$, let $\MAX(V)$ be $V$ equipped with the matrix-norm given by
\[
\|A\|_{\MAX(V),m,n}:=\sup\left\{\left\|M_{m,n}(\phi)(A)\right\|_{\B\left(H^{(n)},H^{(m)}\right)}:\begin{array}{c}H\textrm{ is a Hilbert space},\\ \phi:V\to\B(H)\textrm{ linear},\\ \|\phi\|_{\B(V,\B(H))}\leq 1\end{array}\right\}.
\]
\end{defn}

\begin{defn}[{Absolute maximum matrix-norm structure, \cite[Theorem 2.1]{effros1988}}]\label{AMAX-definition}
Recall that $\mathbb{M}_{n,m}^{*}$ can be identified as $\mathbb{M}_{m,n}$ equipped with the trace norm.  Given a Banach space $V$, let $\AMAX(V)$ be $V$ equipped with the matrix-norm given by
\[
\|A\|_{\AMAX(V),m,n}:=\inf\left\{\sum_{l=1}^{p}\left\|v_{l}\right\|_{V}\left\|C_{l}\right\|_{\mathbb{M}_{n,m}^{*}}:A=\sum_{l=1}^{p}v_{l}\otimes C_{l}\right\},
\]
the projective tensor norm on $V\otimes\mathbb{M}_{n,m}^{*}$
\end{defn}

Please note that each of these constructions is distinct from the others.

\begin{ex}[Distinction between $\MIN$, $\MAX$, and $\AMAX$]
Observe that
\[
\left\|\begin{bmatrix}1 & 0\\ 0 & 1\end{bmatrix}\right\|_{\MIN(\mathbb{C}),2,2}
=\left\|\begin{bmatrix}1 & 0\\ 0 & 1\end{bmatrix}\right\|_{\MAX(\mathbb{C}),2,2}
=1
\neq 2
=\left\|\begin{bmatrix}1 & 0\\ 0 & 1\end{bmatrix}\right\|_{\AMAX(\mathbb{C}),2,2},
\]
so $\AMAX(\mathbb{C})\not\cong_{\MBanC}\MAX(\mathbb{C})$ and $\AMAX(\mathbb{C})\not\cong_{\MBanC}\MIN(\mathbb{C})$.  By \cite[Proposition 2.7]{paulsen1992}, $\MAX\left(\ell^{2}([3])\right)\not\cong_{\MBanC}\MIN\left(\ell^{2}([3])\right)$.
\end{ex}

Of all of the ways a norm can be extended to a matrix-norm, the two most important for the purposes of this paper are $\MIN$ and $\AMAX$ as they are the least and greatest matrix-norm, respectively, which extend the original Banach space norm.  The minimality of $\MIN$ is well-known as stated in \cite[Example 2.3]{blecher1991} and yields the following universal property.

\begin{thm}[Universal property of $\MIN$]\label{MIN-prop}
Let $F_{\MBanB}^{\BanB}:\MBanB\to\BanB$ be the forgetful functor stripping all matrix-norm structure except the underlying norm.  For a Banach space $V$ and a matricial Banach space $W$, consider a bounded linear map $\varphi:F_{\MBanB}^{\BanB}(W)\to V$.  Then, there is a unique completely bounded linear map $\hat{\varphi}:W\to\MIN(V)$ such that $F_{\MBanB}^{\BanB}\left(\hat{\varphi}\right)=\varphi$.  Moreover,
\[
\left\|\hat{\varphi}\right\|_{\CB(W,\MIN(V))}=\|\varphi\|_{\B\left(F_{\MBanB}^{\BanB}(W),V\right)}.
\]
\end{thm}

The proof of the above theorem arises from the same logic as \cite[Exercise 14.1]{paulsen}, except using a general matricial Banach space instead of an operator space.  On the other hand, the maximality of $\AMAX$ will be proven as the author has no knowledge of its proof in the literature.

\begin{lem}[Maximality of $\AMAX$]
Let $W$ be a matricial Banach space.  For $m,n\in\mathbb{N}$ and $A\in M_{m,n}(W)$,
\[
\|A\|_{W,m,n}\leq\|A\|_{\AMAX\left(F_{\MBanB}^{\BanB}(W)\right),m,n}.
\]
\end{lem}

\begin{proof}

For some $w\in W$ and $C\in\mathbb{M}_{m,n}$, let $P:=\left(C^{*}C\right)^{1/2}$ and $U\in\mathbb{M}_{m,n}$ satisfy $C=UP$ as in the polar decomposition.  By the spectral theorem, write $P=\sum_{j=1}^{n}s_{j}\left(h_{j}h_{j}^{*}\right)$ for some orthonormal basis $\left(h_{j}\right)_{j=1}^{n}\subseteq\mathbb{C}^{n}\cong\mathbb{M}_{n,1}$ and positive scalars $\left(s_{j}\right)_{j=1}^{n}$.  Then,
\[\begin{array}{rcl}
\|w\otimes C\|_{W,m,n}
&	=	&	\left\|U\left(w\otimes P\right)\right\|_{W,m,n}\\[12pt]
&	=	&	\left\|U\left(\sum_{j=1}^{n}\left(s_{j}w\right)\otimes\left(h_{j}h_{j}^{*}\right)\right)\right\|_{W,m,n}\\[12pt]
&	\leq	&	\left\|U\right\|_{\mathbb{M}_{m,n}}\sum_{j=1}^{n}\left\|\left(s_{j}w\right)\otimes\left(h_{j}h_{j}^{*}\right)\right\|_{W,n,n}\\[12pt]
&	\leq	&	\left\|U\right\|_{\mathbb{M}_{m,n}}\sum_{j=1}^{n}s_{j}\left\|h_{j}\right\|_{\mathbb{M}_{n,1}}\left\|w\right\|_{W,1,1}\left\|h_{j}^{*}\right\|_{\mathbb{M}_{1,n}}\\[12pt]
\end{array}\]
\[\begin{array}{rcl}
&	=	&	\left\|w\right\|_{F_{\MBanB}^{\BanB}(W)}\sum_{j=1}^{n}s_{j}\\[12pt]
&	=	&	\left\|w\right\|_{F_{\MBanB}^{\BanB}(W)}\|C\|_{\mathbb{M}_{n,m}^{*}}.\\[12pt]
\end{array}\]

If $A=\sum_{l=1}^{p}w_{l}\otimes C_{l}$, then
\[
\|A\|_{W,m,n}
\leq\sum_{l=1}^{p}\left\|w_{l}\otimes C_{l}\right\|_{W,m,n}
\leq\sum_{l=1}^{p}\left\|w_{l}\right\|_{F_{\MBanB}^{\BanB}(W)}\|C_{l}\|_{\mathbb{M}_{n,m}^{*}}.
\]
Taking the infimum over all ways of representing $A$ gives the result.

\end{proof}

Likewise, $\AMAX$ gains a universal property from its extremal nature as well.

\begin{thm}[Universal property of $\AMAX$]\label{AMAX-prop}
For a Banach space $V$ and a matricial Banach space $W$, consider a bounded linear map $\phi:V\to F_{\MBanB}^{\BanB}(W)$.  Then, there is a unique completely bounded linear map $\hat{\phi}:\AMAX(V)\to W$ such that $F_{\MBanB}^{\BanB}\left(\hat{\phi}\right)=\phi$.  Moreover,
\[
\left\|\hat{\phi}\right\|_{\CB(\AMAX(V),W)}=\|\phi\|_{\B\left(V,F_{\MBanB}^{\BanB}(W)\right)}.
\]
\end{thm}

Again, the proof of the above theorem is nearly identical to \cite[Exercise 14.1]{paulsen}, except using a general matricial Banach space instead of an operator space.  It is of note that $\MAX$ has a universal property almost identical to $\AMAX$, except that the target space for $\MAX$ must be an abstract operator space as defined in \cite[p.\ 184]{paulsen}.

\subsection{Constructions for matricial Banach spaces}\label{matrix-normed-constructions}

For Sections \ref{constructions} and \ref{haageruptensoralgebra}, some important constructions for matricial Banach spaces will be presented.  First, if a matrix-normed space is not complete, the matrix-norms may be extended naturally to the metric completion.

For abstract operator spaces as defined in \cite[p.\ 184]{paulsen}, completions are trivial since by \cite[Theorem 3.1]{ruan1988}, an abstract operator space is completely isometrically isomorphic to a subspace of operators on a Hilbert space.  Thus, the completion can be done in the space of operators.  However, this result does not apply to more general matrix-normed spaces.  As such, this result will be done in detail.

\begin{defn}[Notation for the completion]
Given a normed space $V$, let $\C(V)$ be the completion of $V$ and $\kappa_{V}:V\to \C(V)$ the canonical embedding of $V$ into $\C(V)$.
\end{defn}

\begin{lem}[Limits of matrix-norms]\label{lemma-completion}
Let $V$ be a matrix-normed space, $m,n\in\mathbb{N}$, and $A\in M_{m,n}\left(\C\left(F_{\MBanB}^{\BanB}(V)\right)\right)$.  Consider sequences $\left(B_{p}\right)_{p\in\mathbb{N}},\left(C_{p}\right)_{p\in\mathbb{N}}\subseteq M_{m,n}(V)$ such that
\[
\lim_{p\to\infty}\kappa_{V}\left(B_{p}(j,k)\right)
=
\lim_{p\to\infty}\kappa_{V}\left(C_{p}(j,k)\right)
=A(j,k)
\]
for all $1\leq j\leq m$ and $1\leq k\leq n$.  Then, the limits $\lim_{p\to\infty}\left\|B_{p}\right\|_{V,m,n}$ and $\lim_{p\to\infty}\left\|C_{p}\right\|_{V,m,n}$ converge and are equal.
\end{lem}

\begin{proof}

Let $\epsilon>0$.  For $1\leq j\leq m$ and $1\leq k\leq n$, there is $N_{j,k}\in\mathbb{N}$ such that if $p\geq N_{j,k}$, then $\left\|\kappa_{V}\left(B_{p}(j,k)\right)-A(j,k)\right\|_{\C\left(F_{\MBanB}^{\BanB}(V)\right)}<\frac{\epsilon}{2mn}$.  Choose $N:=\max\left\{N_{j,k}:1\leq j\leq m,1\leq k\leq n\right\}$.  For $p,q\geq N$, \cite[Proposition 2.1]{ruan1988} gives
\[
\left|\left\|B_{p}\right\|_{V,m,n}-\left\|B_{q}\right\|_{V,m,n}\right|
\leq	\left\|B_{p}-B_{q}\right\|_{V,m,n}\\[12pt]
\]
\[\begin{array}{rcl}
&	\leq	&	\sum_{j=1}^{m}\sum_{k=1}^{n}\left\|B_{p}(j,k)-B_{q}(j,k)\right\|_{V,1,1}\\[12pt]

&	=	&	\sum_{j=1}^{m}\sum_{k=1}^{n}\left\|\kappa_{V}\left(B_{p}(j,k)\right)-\kappa_{V}\left(B_{q}(j,k)\right)\right\|_{\C\left(F_{\MBanB}^{\BanB}(V)\right)}\\[12pt]

&	\leq	&	\sum_{j=1}^{m}\sum_{k=1}^{n}\left\|\kappa_{V}\left(B_{p}(j,k)\right)-A(j,k)\right\|_{\C\left(F_{\MBanB}^{\BanB}(V)\right)}\\
&	&	+\sum_{j=1}^{m}\sum_{k=1}^{n}\left\|A(j,k)-\kappa_{V}\left(B_{q}(j,k)\right)\right\|_{\C\left(F_{\MBanB}^{\BanB}(V)\right)}\\[12pt]

&	<	&	\sum_{j=1}^{m}\sum_{k=1}^{n}\frac{\epsilon}{2mn}+\sum_{j=1}^{m}\sum_{j=1}^{n}\frac{\epsilon}{2mn}\\[12pt]
&	=	&	\epsilon.\\[12pt]
\end{array}\]
The sequence $\left(\left\|B_{p}\right\|_{V,m,n}\right)_{p=1}^{\infty}$ is Cauchy and, therefore, convergent.  A similar argument shows $\left(\left\|C_{p}\right\|_{V,m,n}\right)_{p=1}^{\infty}$ is also convergent.  Then,
\[
\left|\lim_{p\to\infty}\left\|B_{p}\right\|_{V,m,n}-\lim_{p\to\infty}\left\|C_{p}\right\|_{V,m,n}\right|
\leq	\lim_{p\to\infty}\left\|B_{p}-C_{p}\right\|_{V,m,n}
\]
\[\begin{array}{rcl}
&	\leq	&	\sum_{j=1}^{m}\sum_{j=1}^{n}\lim_{p\to\infty}\left\|B_{p}(j,k)-C_{p}(j,k)\right\|_{V,1,1}\\[12pt]
&	\leq	&	\sum_{j=1}^{m}\sum_{j=1}^{n}\lim_{p\to\infty}\left\|\kappa_{V}\left(B_{p}(j,k)\right)-\kappa_{V}\left(C_{p}(j,k)\right)\right\|_{\C\left(F_{\MBanB}^{\BanB}(V)\right)}\\[12pt]
&	=	&	0.\\[12pt]
\end{array}\]

\end{proof}

\begin{defn}[Matricial completion]
Given a matrix-normed space $V$, let
\[
\MC(V):=\C\left(F_{\MBanB}^{\BanB}(V)\right)
\]
equipped with the functions defined by
\[
\|A\|_{\MC(V),m,n}:=\lim_{p\to\infty}\left\|B_{p}\right\|_{V,m,n},
\]
where $\left(B_{p}\right)_{p\in\mathbb{N}}\subseteq M_{m,n}(V)$ is any sequence satisfying
\[
\lim_{p\to\infty}\kappa_{V}\left(B_{p}(j,k)\right)=A(j,k).
\]
for all $1\leq j\leq m$ and $1\leq k\leq n$.  By Lemma \ref{lemma-completion}, this definition is unambiguous.
\end{defn}

\begin{lem}\label{matrix-norm-complete}
Equipped with the above functions, $\MC(V)$ is matricial Banach space such that
\[
\left\|M_{m,n}\left(\kappa_{V}\right)(A)\right\|_{\MC(V),m,n}=\|A\|_{V,m,n}
\]
for all $m,n\in\mathbb{N}$ and $A\in M_{m,n}(V)$.
\end{lem}

\begin{proof}

Let $m,n,j,k\in\mathbb{N}$, $A,B\in M_{m,n}(\MC(V))$, $C\in\mathbb{M}_{j,m}$, and $D\in\mathbb{M}_{n,k}$.  For a sequence $\left(A_{p}\right)_{p\in\mathbb{N}}\subseteq M_{m,n}(V)$ that point-wise converges to $A$, note that $CA_{p}D$ point-wise converges to $CAD$ also.  Thus,
\[\begin{array}{rcl}
\|CAD\|_{\MC(V),j,k}
&	=	&	\lim_{p\to\infty}\left\|CA_{p}D\right\|_{V,j,k}\\[12pt]
&	\leq	&	\lim_{p\to\infty}\left\|C\right\|_{\mathbb{M}_{j,m}}\left\|A_{p}\right\|_{V,m,n}\left\|D\right\|_{\mathbb{M}_{n,k}}\\[12pt]
&	=	&	\left\|C\right\|_{\mathbb{M}_{j,m}}\left\|A\right\|_{\MC(V),m,n}\left\|D\right\|_{\mathbb{M}_{n,k}}.\\[12pt]
\end{array}\]

For a sequence $\left(B_{p}\right)_{p\in\mathbb{N}}\subseteq M_{m,n}(V)$ that point-wise converges to $B$, note that $A_{p}+B_{p}$ point-wise converges to $A+B$.  Thus,
\[\begin{array}{rcl}
\|A+B\|_{\MC(V),m,n}
&	=	&	\lim_{p\to\infty}\left\|A_{p}+B_{p}\right\|_{V,m,n}\\[12pt]
&	\leq	&	\lim_{p\to\infty}\left\|A_{p}\right\|_{V,m,n}+\lim_{p\to\infty}\left\|B_{p}\right\|_{V,m,n}\\[12pt]
&	=	&	\left\|A\right\|_{\MC(V),m,n}+\left\|B\right\|_{\MC(V),m,n}.\\[12pt]
\end{array}\]

If $A\in M_{m,n}(V)$, then the constant sequence $A_{p}:=A$ converges point-wise to $M_{m,n}\left(\kappa_{V}\right)(A)$, so $\left\|M_{m,n}\left(\kappa_{V}\right)(A)\right\|_{\MC(V),m,n}=\|A\|_{V,m,n}$.  As the underlying normed space of $\MC(V)$ is $\C\left(F_{\MBanB}^{\BanB}(V)\right)$, $\MC(V)$ is a matricial Banach space.

\end{proof}

As a result, the canonical embedding $\kappa_{V}$ is completely isometric.  Moreover, $\MC(V)$ has the following universal property, analogous to the universal property of the metric completion in \cite[Example I.4.17C(8)]{joyofcats}.

\begin{thm}[Universal property of the matricial completion]\label{univ-prop-complete}
Given a matricial Banach space $W$ and a completely bounded linear map $\phi:V\to W$, there is a unique completely bounded linear map $\hat{\phi}:\MC(V)\to W$ such that $\hat{\phi}\circ\kappa_{V}=\phi$.  Moreover,
\[
\left\|\hat{\phi}\right\|_{\CB\left(\MC(V),W\right)}=\|\phi\|_{\CB\left(V,W\right)}
\]
\end{thm}

The proof of the theorem proceeds identically to the normed space case.  The second construction is the extension of the $\ell^{1}$-direct sum of Banach spaces.

\begin{defn}[Matricial $\ell^{1}$-direct sum]
Given an index set $\Lambda$, let $\left(V_{\lambda}\right)_{\lambda\in\Lambda}$ be matricial Banach spaces.  Define
\[
V:={\coprod_{\lambda\in\Lambda}}^{\BanC}F_{\MBanB}^{\BanB}\left(V_{\lambda}\right),
\]
the $\ell^{1}$-direct sum of the underlying Banach spaces.  For $\lambda\in\Lambda$, let $\varpi_{\lambda}:V_{\lambda}\to V$ be the canonical inclusion, and $\pi_{\lambda}:V\to V_{\lambda}$ the canonical projection.  Define norm functions
\[
\|A\|_{V,m,n}:=\sum_{\lambda\in\Lambda}\left\|M_{m,n}\left(\pi_{\lambda}\right)(A)\right\|_{V_{\lambda},m,n},
\]
each an $\ell^{1}$-sum norm.  One can check that these norms constitute a matrix-norm on $V$.  Equipped with this matrix-norm, $V$ is the \emph{matricial $\ell^{1}$-direct sum} of the $V_{\lambda}$.
\end{defn}

Much like the $\ell^{1}$-direct sum of Banach spaces, the matricial $\ell^{1}$-direct sum of matricial Banach spaces has a weakened version of the coproduct universal property.

\begin{thm}[Universal property of the coproduct, $\MBanB$]\label{MBan-coproduct}
For a matricial Banach space $W$, let $\phi_{\lambda}:V_{\lambda}\to W$ be completely bounded linear maps satisfying
\[
\sup\left\{\left\|\phi_{\lambda}\right\|_{\CB\left(V_{\lambda},W\right)}:\lambda\in\Lambda\right\}<\infty.
\]
There is a unique completely bounded linear map $\phi:V\to W$ such that $\phi\circ\varpi_{\lambda}=\phi_{\lambda}$ for all $\lambda\in\Lambda$.  Moreover,
\[
\|\phi\|_{\CB\left(V,W\right)}=\sup\left\{\left\|\phi_{\lambda}\right\|_{\CB\left(V_{\lambda},W\right)}:\lambda\in\Lambda\right\}.
\]
\end{thm}

The proof of the above theorem mirrors its Banach space counterpart in \cite[Example 2.2.4.h]{borceux1}.  Moreover, this theorem guarantees that $\MBanC$ has all coproducts.  As such, the notation
\[
{\coprod_{\lambda\in\Lambda}}^{\MBanC}V_{\lambda}
\]
will be used to denote the matricial $\ell^{1}$-direct sum of the family $\left(V_{\lambda}\right)_{\lambda\in\Lambda}$.  Also, be aware that this coproduct is not the coproduct of operator spaces from \cite[p.\ 269]{oikhberg1999}, even when the summands are operator spaces as shown in the example below.

\begin{ex}[Distinction between coproducts]
Letting
\[
V:=\MIN(\mathbb{C}){\coprod}^{\MBanC}\MIN(\mathbb{C}),
\]
observe that
\[
\left\|\begin{bmatrix}(1,0) & (0,0)\\ (0,0) & (0,1)\end{bmatrix}\right\|_{V,2,2}
=2
\neq 1
=\max\left\{\|(1,0)\|_{V,1,1},\|(0,1)\|_{V,1,1}\right\},
\]
Hence, $V$ is not an abstract operator space in the sense of \cite[p.\ 184]{paulsen}.  Consequently, $V$ is not the operator space coproduct of $\MIN(\mathbb{C})$ with itself.  Moreover, this means that
\[
\MIN\left(\mathbb{C}{\coprod}^{\BanC}\mathbb{C}\right)\not\cong_{\MBanC}\MIN(\mathbb{C}){\coprod}^{\MBanC}\MIN(\mathbb{C}).
\]
\end{ex}

However, $\AMAX$ will preserve coproducts.

\begin{cor}[$\AMAX$ and direct sums]
Given an index set $\Lambda$, let $\left(V_{\lambda}\right)_{\lambda\in\Lambda}$ be Banach spaces.  Then,
\[
\AMAX\left({\coprod_{\lambda\in\Lambda}}^{\BanC}V_{\lambda}\right)
\cong_{\MBanC}
{\coprod_{\lambda\in\Lambda}}^{\MBanC}\AMAX\left(V_{\lambda}\right).
\]
\end{cor}

The proof follows immediately as $\AMAX$ is a left adjoint functor.  The final construction is the analogue of the projective tensor product.

\begin{defn}[{Haagerup tensor product, \cite[\S 3]{blecher1991}}]
Given matricial Banach spaces $V$ and $W$, let $m,n,p\in\mathbb{N}$, $A\in M_{m,p}(V)$, and $B\in M_{p,n}(W)$.  Define the \emph{tensor matrix product} $A\odot B\in M_{m,n}(V\otimes W)$ of $A$ and $B$ entrywise by
\[
(A\odot B)(i,k):=\sum_{j=1}^{p}A(i,j)\otimes B(j,k).
\]
Define the \emph{Haagerup matrix-norm} on $V\otimes W$ by
\[
\|C\|_{V\otimes_{h}W,m,n}:=
\inf\left\{\sum_{l=1}^{p}\left\|A_{l}\right\|_{V,m,q_{l}}\left\|B_{l}\right\|_{W,q_{l},n}: C=\sum_{l=1}^{p}A_{l}\odot B_{l}\right\}.
\]
Let $V\otimes_{h}W$ denote $V\otimes W$ equipped with this matrix-norm and completed into a matricial Banach space, the \emph{Haagerup tensor product} of $V$ and $W$.
\end{defn}

Similar to the projective tensor product, the Haagerup tensor product has a universal property when dealing with a class of bilinear maps.

\begin{defn}[{Completely bounded bilinear maps, \cite[p.\ 250]{paulsen}}]
Given matricial Banach spaces $V$ and $W$, let $m,n,p\in\mathbb{N}$, $A\in M_{m,p}(V)$, and $B\in M_{p,n}(W)$.  For a bilinear map $\phi:V\times W\to Z$, define the \emph{$\phi$-matrix product} $A\odot_{\phi}B\in M_{m,n}(Z)$ of $A$ and $B$ entrywise by
\[
(A\odot_{\phi}B)(i,k):=\sum_{j=1}^{p}\phi\left(A(i,j),B(j,k)\right).
\]
The map $\phi$ is \emph{completely bounded} if there is $L\geq0$ such that
\[
\left\|A\odot_{\phi}B\right\|_{Z,m,n}\leq L\cdot\|A\|_{V,m,p}\|B\|_{W,p,n}
\]
for all $m,n,p\in\mathbb{N}$, $A\in M_{m,p}(V)$, and $B\in M_{p,n}(W)$.  Let
\[
\|\phi\|_{\CB(V,W;Z)}:=\inf\left\{L\in[0,\infty):\begin{array}{c}\left\|A\odot_{\phi}B\right\|_{Z,m,n}\leq L\cdot\|A\|_{V,m,p}\|B\|_{W,p,n}\\ \forall m,n,p\in\mathbb{N}, A\in M_{m,p}(V), B\in M_{p,n}(W)\end{array}\right\}.
\]
\end{defn}

\begin{thm}[{Universal property of $\otimes_{h}$, \cite[Exercise 17.3]{paulsen}}]\label{haagerup}
Given matricial Banach spaces $V$, $W$, and $Z$, consider a completely bounded bilinear map $\phi:V\times W\to Z$.  There is a unique completely bounded linear map $\hat{\phi}:V\otimes_{h}W\to Z$ such that $\hat{\phi}(v\otimes w)=\phi(v,w)$ for all $v\in V$ and $w\in W$.  Moreover,
\[
\left\|\hat{\phi}\right\|_{\CB\left(V\otimes_{h}W,Z\right)}
=
\left\|\phi\right\|_{\CB\left(V,W;Z\right)}.
\]
\end{thm}

By \cite[Proposition 3.1]{blecher1991}, $\otimes_{h}$ is associative.  Moreover, $\otimes_{h}$ interacts well with $\MIN(\mathbb{C})$, $\AMAX$, and the projective tensor product of Banach spaces.  The first result shows that $\MIN(\mathbb{C})$ acts as an identity for $\otimes_{h}$.

\begin{prop}[Unit object of $\otimes_{h}$]
For a matricial Banach space $V$,
\[
V\otimes_{h}\MIN(\mathbb{C})
\cong_{\MBanC}\MIN(\mathbb{C})\otimes_{h}V
\cong_{\MBanC}V.
\]
\end{prop}

The proof of the above proposition is showing the canonical maps $\lambda\otimes v\mapsto\lambda v$ and $v\otimes\lambda\mapsto\lambda v$ are completely isometric, which follow readily from direct computation.  Tedious calculations can show that $\otimes_{h}$ is a monoidal product on the categories $\MBanB$ and $\MBanC$.

Lastly, $\AMAX$ actually converts the projective tensor into the Haagerup tensor.  This will be proven as the author has no knowledge of its proof in the literature.

\begin{thm}[$\AMAX$, $\hat{\otimes}$, $\otimes_{h}$]
Given Banach spaces $V$ and $W$,
\[
\AMAX\left(V\hat{\otimes}W\right)
\cong_{\MBanC}
\AMAX(V)\otimes_{h}\AMAX(W).
\]
\end{thm}

\begin{proof}

Define $\phi:\AMAX(V)\times\AMAX(W)\to\AMAX\left(V\hat{\otimes}W\right)$ by $\phi(v,w):=v\otimes w$.  This map is quickly seen to be bilinear, so it remains to show it completely bounded.  For $m,n,p\in\mathbb{N}$, let $A\in M_{m,p}(V)$ and $B\in M_{p,n}(W)$.  Write $A=\sum_{t=1}^{q}v_{t}\otimes C_{t}$ and $B=\sum_{s=1}^{r}w_{s}\otimes D_{s}$ for $v_{t}\in V$, $w_{s}\in W$, $C_{t}\in\mathbb{M}_{m,p}$, and $D_{s}\in\mathbb{M}_{p,n}$.  Observe that
\[
A\odot_{\phi}B=\sum_{t=1}^{q}\sum_{s=1}^{r}\left(v_{t}\otimes w_{s}\right)\otimes\left(C_{t}D_{s}\right)
\]
by the bilinearity of $\phi$.  Thus,
\[\begin{array}{rcl}
\left\|A\odot_{\phi}B\right\|_{\AMAX\left(V\hat{\otimes}W\right),m,n}
&	\leq	&	\sum_{t=1}^{q}\sum_{s=1}^{r}\left\|v_{t}\otimes w_{s}\right\|_{V\hat{\otimes}W}\left\|C_{t}D_{s}\right\|_{\mathbb{M}_{n,m}^{*}}\\[15pt]
&	=	&	\sum_{t=1}^{q}\sum_{s=1}^{r}\left\|v_{t}\right\|_{V}\left\|w_{s}\right\|_{W}\left\|C_{t}D_{s}\right\|_{\mathbb{M}_{n,m}^{*}}\\[15pt]
&	\leq	&	\sum_{t=1}^{q}\sum_{s=1}^{r}\left\|v_{t}\right\|_{V}\left\|w_{s}\right\|_{W}\left\|C_{t}\right\|_{\mathbb{M}_{p,m}^{*}}\left\|D_{s}\right\|_{\mathbb{M}_{n,p}^{*}}\\[15pt]
&	=	&	\left(\sum_{t=1}^{q}\left\|v_{t}\right\|_{V}\left\|C_{t}\right\|_{\mathbb{M}_{p,m}^{*}}\right)\left(\sum_{s=1}^{r}\left\|w_{s}\right\|_{W}\left\|D_{s}\right\|_{\mathbb{M}_{n,p}^{*}}\right).\\[15pt]
\end{array}\]
Taking infima yields
\[
\left\|A\odot_{\phi}B\right\|_{\AMAX\left(V\hat{\otimes}W\right),m,n}
\leq\|A\|_{\AMAX(V),m,p}\|B\|_{\AMAX(W),p,n}.
\]
By Theorem \ref{haagerup}, there is a unique completely contractive linear map $\hat{\phi}:\AMAX(V)\otimes_{h}\AMAX(W)\to\AMAX\left(V\hat{\otimes}W\right)$ such that $\hat{\phi}(v\otimes w)=\phi(v,w)=v\otimes w$.

Define $\varphi:V\times W\to F_{\MBanB}^{\BanB}\left(\AMAX(V)\otimes_{h}\AMAX(W)\right)$ by $\varphi(v,w):=v\otimes w$.  This map is quickly seen to be bilinear, so it remains to show it bounded.  For $v\in V$ and $w\in W$,
\[\begin{array}{rcl}
\|\varphi(v,w)\|_{\AMAX(V)\otimes_{h}\AMAX(W),1,1}
&	=	&	\|v\otimes w\|_{\AMAX(V)\otimes_{h}\AMAX(W),1,1}\\[8pt]
&	=	&	\|v\|_{\AMAX(V),1,1}\cdot\|w\|_{\AMAX(W),1,1}\\[8pt]
&	=	&	\|v\|_{V}\cdot\|w\|_{W}.\\[8pt]
\end{array}\]
By the universal property of the projective tensor product, there is a unique contractive linear map $\hat{\varphi}:V\hat{\otimes}W\to F_{\MBanB}^{\BanB}\left(\AMAX(V)\otimes_{h}\AMAX(W)\right)$ such that $\hat{\varphi}(v\otimes w)=\varphi(v,w)=v\otimes w$.  By Theorem \ref{AMAX-prop}, there is a unique completely contractive linear map $\tilde{\varphi}:\AMAX\left(V\hat{\otimes}W\right)\to\AMAX(V)\otimes_{h}\AMAX(W)$ such that $F_{\MBanB}^{\BanB}\left(\tilde{\varphi}\right)=\hat{\varphi}$.

Immediate calculations show that
\[\begin{array}{ccc}
\left(\tilde{\varphi}\circ\hat{\phi}\right)(v\otimes w)=v\otimes w
& \textrm{and} &
\left(\hat{\phi}\circ\tilde{\varphi}\right)(v\otimes w)=v\otimes w,
\end{array}\]
for all $v\in V$ and $w\in W$.  By the universal properties of $\hat{\otimes}$, $\AMAX$, and $\otimes_{h}$,
\[\begin{array}{ccc}
\tilde{\varphi}\circ\hat{\phi}=id_{\AMAX(V)\otimes_{h}\AMAX(W)}
& \textrm{and} &
\hat{\phi}\circ\tilde{\varphi}=id_{\AMAX\left(V\hat{\otimes}W\right)}.
\end{array}\]

\end{proof}

\section{Array-Weighted Sets}\label{awsets}

This section introduces a new category of objects for the construction of a scaled-free matrix-normed space.  The content of this section is based heavily on the results in \cite[\S 2]{grilliette1} and can be considered an extension of both \cite[\S 1.1]{gerbracht} and \cite[\S 2.2]{grandis2004}.

As with Banach spaces, the forgetful functor from $\MBanB$ to $\Set$ stripping all structure will not have a left adjoint, meaning there is no free matricial Banach space.  Instead, one could consider the forgetful functor from $\MBanB$ to $\WSetB$, where all structure is dropped except for the norm and the underlying set.  However, due to \cite[Theorem 3.1.1]{grilliette1} and Theorem \ref{AMAX-prop}, closure of left adjoints states that the left adjoint must be $\AMAX\circ\BanSp$.  Consequently, the absolute maximum matrix-norm is imposed, which does not allow tighter bounds on the matrix-norm beyond the underlying normed space.

The objects defined in Section \ref{definitions} will remedy this issue through an ``array-weight'', which will allow finer control for the object built in Section \ref{constructions}.  Section \ref{max-min} produces two extremal ways of extending a weight function on a set to an array-weight, much like extending a norm to a matrix-norm.  Section \ref{zero-weight} describes a minimal way of appending an element with weight value 0 to an existing array-weighted set, which is useful in building an array-weight on the disjoint union in Section \ref{set-constructions}.  Section \ref{array-free} discusses maps of these array-weighted sets into $\MIN(\mathbb{C})$, which will have an effect on linear independence of generators in Section \ref{constructions}.

\subsection{Definitions and Basic Results}\label{definitions}

To motivate the main definition of this section, consider the following two properties of a matrix-normed space.  The norm of a matrix is bounded below by any compression or rearrangement of rows and columns.  Likewise, the norm of a matrix is bounded above by the sum of the norms of its blocks.  The following example illustrates these two properties explicitly.

\begin{ex}\label{motivating}
Let $V$ be a matrix-normed space.  For integers $1\leq j\leq m$ and a one-to-one function $\alpha:[j]\to[m]$, define the isometry $U_{\alpha}:\mathbb{C}^{j}\to\mathbb{C}^{m}$ on the standard basis by $U_{\alpha}\left(e_{a}\right):=e_{\alpha(a)}$.

For integers $1\leq j\leq m$ and $1\leq k\leq n$, let $\alpha:[j]\to[m]$ and $\beta:[k]\to[n]$ be one-to-one, and $\alpha\times\beta:[j]\times[k]\to[m]\times[n]$ be the cartesian product map.  From Definition \ref{matrix-notation}, recall that $A\in M_{m,n}(V)$ is fundamentally a function from $[m]\times[n]$ to $V$, so the composition $A\circ(\alpha\times\beta)\in M_{j,k}(V)$ would be defined entrywise by
\[
\left(A\circ(\alpha\times\beta)\right)(a,b)=A(\alpha(a),\beta(b))
\]
for all $1\leq a\leq j$ and $1\leq b\leq k$.  A quick calculation shows that $A\circ(\alpha\times\beta)=U_{\alpha}^{*}AU_{\beta}$, so
\[
\left\|A\circ(\alpha\times\beta)\right\|_{V,j,k}
=\left\|U_{\alpha}^{*}AU_{\beta}\right\|_{V,j,k}
\leq\left\|U_{\alpha}^{*}\right\|_{\mathbb{M}_{j,m}}\|A\|_{V,m,n}\left\|U_{\beta}\right\|_{\mathbb{M}_{n,k}}
=\|A\|_{V,m,n}.
\]

In the case $j<m$, let $\gamma:[m-j]\to[m]$ be one-to-one such that $\Ran(\alpha)\cap\Ran(\gamma)=\emptyset$.  Then, the identity $I_{m}$ of $\mathbb{M}_{m,m}$ can be written as $U_{\alpha}U_{\alpha}^{*}+U_{\gamma}U_{\gamma}^{*}$.  Thus,
\[\begin{array}{rcl}
\|A\|_{V,m,n}
&	=	&	\left\|I_{m}A\right\|_{V,m,n}\\[12pt]
&	=	&	\left\|\left(U_{\alpha}U_{\alpha}^{*}+U_{\gamma}U_{\gamma}^{*}\right)A\right\|_{V,m,n}\\[12pt]
&	=	&	\left\|U_{\alpha}U_{\alpha}^{*}A+U_{\gamma}U_{\gamma}^{*}A\right\|_{V,m,n}\\[12pt]
&	\leq	&	\left\|U_{\alpha}\right\|_{\mathbb{M}_{m,j}}\left\|U_{\alpha}^{*}A\right\|_{V,j,n}+\left\|U_{\gamma}\right\|_{\mathbb{M}_{m,m-j}}\left\|U_{\gamma}^{*}A\right\|_{V,m-j,n}\\[12pt]
&	=	&	\left\|U_{\alpha}^{*}A\right\|_{V,j,n}+\left\|U_{\gamma}^{*}A\right\|_{V,m-j,n}\\[12pt]
&	=	&	\left\|A\circ\left(\alpha\times id_{[n]}\right)\right\|_{V,j,n}+\left\|A\circ\left(\gamma\times id_{[n]}\right)\right\|_{V,m-j,n}.\\[12pt]
\end{array}\]

In the case $k<n$, an identical calculation shows that
\[
\|A\|_{V,m,n}\leq\left\|A\circ\left(id_{[m]}\times\beta\right)\right\|_{V,m,k}+\left\|A\circ\left(id_{[m]}\times\delta\right)\right\|_{V,m,n-k}
\]
for $\delta:[n-k]\to[n]$ one-to-one satisfying $\Ran(\beta)\cap\Ran(\delta)=\emptyset$.
\end{ex}

This interplay between the norms is the core notion for the main definition.  However, since an arbitrary set need not have an action of $\mathbb{C}$ upon it, matrix multiplication will be replaced with function composition.

\begin{defn}
For a set $X$, an \emph{array-weight} on $X$ is a net $\left(w_{X,m,n}\right)_{m,n\in\mathbb{N}}$ such that
\begin{enumerate}
\item $w_{X,m,n}:M_{m,n}(X)\to[0,\infty)$ for all $m,n\in\mathbb{N}$,
\item $w_{X,j,k}\left(A\circ(\alpha\times\beta)\right)\leq w_{X,m,n}(A)$ for all $1\leq j\leq m$, $1\leq k\leq n$, $A\in M_{m,n}(X)$, and one-to-one functions $\alpha:[j]\to[m]$ and $\beta:[k]\to[n]$,
\item $w_{X,m,n}(A)\leq w_{X,j,n}\left(A\circ\left(\alpha\times id_{[n]}\right)\right)+w_{X,m-j,n}\left(A\circ\left(\gamma\times id_{[n]}\right)\right)$ for all $1\leq j<m$ and one-to-one functions $\alpha:[j]\to[m]$ and $\gamma:[m-j]\to[m]$ satisfying $\Ran(\alpha)\cap\Ran(\gamma)=\emptyset$,
\item $w_{X,m,n}(A)\leq w_{X,m,k}\left(A\circ\left(id_{[m]}\times\beta\right)\right)+w_{X,m,n-k}\left(A\circ\left(id_{[m]}\times\delta\right)\right)$ for all $1\leq k<n$ and one-to-one functions $\beta:[k]\to[n]$ and $\delta:[n-k]\to[n]$ satisfying $\Ran(\beta)\cap\Ran(\delta)=\emptyset$.
\end{enumerate}
A set equipped with such an array-weight is an \emph{array-weighted set}.
\end{defn}

By Example \ref{motivating}, every matrix-normed space is an array-weighted set when stripped of its linear structure.  Similarly, maps between array-weighted sets are motivated by those between matrix-normed spaces.

\begin{defn}
Given two array-weighted sets $X$ and $Y$, a function $\phi:X\to Y$ is \emph{completely bounded} if there is $L\geq 0$ such that for all $m,n\in\mathbb{N}$ and $A\in M_{m,n}(X)$, $w_{Y,m,n}\left(M_{m,n}(\phi)(A)\right)\leq L\cdot w_{X,m,n}(A)$.  Let
\[
\cbnd(\phi):=\inf\left\{L\in[0,\infty):\begin{array}{c}w_{Y,m,n}\left(M_{m,n}(\phi)(A)\right)\leq L\cdot w_{X,m,n}(A)\\ \forall m,n\in\mathbb{N},A\in M_{m,n}(X)\end{array}\right\},
\]
the \emph{complete bound constant} of $\phi$.  If $\cbnd(\phi)\leq 1$, $\phi$ is \emph{completely contractive}.
\end{defn}

Adaptions of the usual functional analysis proofs yield the following foundational results.

\begin{prop}[Complete-boundedness criteria]\label{complete-bound}
Given array-weighted sets $X$ and $Y$, consider a function $\phi:X\to Y$.  The following are equivalent:
\begin{enumerate}
\item the function $\phi$ is completely bounded;
\item the ampliated function $M_{m,n}(\phi):M_{m,n}(X)\to M_{m,n}(Y)$ is bounded for all $m,n\in\mathbb{N}$ and
\[
\sup\left\{\bnd\left(M_{m,n}(\phi)\right):m,n\in\mathbb{N}\right\}<\infty;
\]
\item $w_{Y,m,n}\left(M_{m,n}(\phi)(A)\right)=0$ for all $m,n\in\mathbb{N}$ and $A\in M_{m,n}(X)$ satisfying $w_{X,m,n}(A)=0$, and
\[
\sup\left(\left\{\frac{w_{Y,m,n}\left(M_{m,n}(\phi)(A)\right)}{w_{X,m,n}(A)}:\begin{array}{c}m,n\in\mathbb{N},A\in M_{m,n}(X),\\ w_{X,m,n}(A)\neq 0\end{array}\right\}\cup\{0\}\right)<\infty.
\]
\end{enumerate}
In this case, $\cbnd(\phi)$ agrees with both suprema and
\[
w_{Y,m,n}\left(M_{m,n}(\phi)(A)\right)\leq \cbnd(\phi)w_{X,m,n}(A)
\]
for all $m,n\in\mathbb{N}$ and $A\in M_{m,n}(X)$
\end{prop}

\begin{cor}[Composition]
Let $X$, $Y$, and $Z$ be array-weighted sets and $\phi:X\to Y$ and $\psi:Y\to Z$ be completely bounded functions.  Then, $\psi\circ\phi:X\to Z$ is completely bounded and
\[
\cbnd(\psi\circ\phi)\leq\cbnd(\psi)\cbnd(\phi).
\]
If $\phi$ and $\psi$ are completely contractive, so is $\psi\circ\phi$.
\end{cor}

\subsection{Maximum and Minimum Array-Weight Structures}\label{max-min}

Given an array-weighted set $X$, $X$ is a weighted set when stripped of all its weight functions, except for the underlying weight function on $M_{1,1}(X)\cong X$.  Given a weighted set, the weight function can be extended to an array-weight in two extremal ways, just as with matrix-normed spaces in Examples \ref{MIN-definition} and \ref{AMAX-definition}.

\begin{defn}[Minimum array-weight structure]
Given a weighted set $S$, let $\mA(S)$ be $S$ equipped with the weight functions
\[
w_{\mA(S),m,n}(A):=\max\left\{w_{S}\left(A(j,k)\right):1\leq j\leq m,1\leq k\leq n\right\},
\]
the maximum weight of an entry in $A$.
\end{defn}

\begin{defn}[Maximum array-weight structure]
Given a weighted set $S$, let $\MA(S)$ be $S$ equipped with the weight functions
\[
w_{\MA(S),m,n}(A):=\sum_{j=1}^{m}\sum_{k=1}^{n}w_{S}\left(A(j,k)\right),
\]
the sum of the weights of the entries in $A$.
\end{defn}

Routine calculations show that each of these nets of weight functions constitute array-weights on $S$ and $w_{\mA(S),1,1}(s)=w_{\MA(S),1,1}(s)=w_{S}(s)$ for all $s\in S$.  Moreover, $\mA$ and $\MA$ are, respectively, the least and greatest array-weight that agree with the original weight function.  The proofs of these two facts follow from inductive use of the definition of an array-weight and reflect the proof of \cite[Proposition 2.1]{ruan1988}.  As a direct result of this optimality, $\mA$ and $\MA$ have the following universal properties, reflecting the universal properties of $\MIN$ and $\AMAX$.

\begin{thm}[Universal property of $\mA$]
Let $\AWSetB$ be the category of array-weighted sets with completely bounded maps, and $F_{\AWSetB}^{\WSetB}:\AWSetB\to\WSetB$ the forgetful functor stripping all weight functions except the underlying weight function.  For a weighted set $S$ and an array-weighted set $Y$, consider a bounded function $\phi:F_{\AWSetB}^{\WSetB}(Y)\to S$.  Then, there is a unique completely bounded map $\hat{\phi}:Y\to\mA(S)$ such that $F_{\AWSetB}^{\WSetB}\left(\hat{\phi}\right)=\phi$.  Moreover,
\[
\cbnd\left(\hat{\phi}\right)=\bnd(\phi).
\]
\end{thm}

\begin{thm}[Universal property of $\MA$]\label{MA-prop}
For a weighted set $S$ and an array-weighted set $Y$, consider a bounded function $\phi:S\to F_{\AWSetB}^{\WSetB}(Y)$.  Then, there is a unique completely bounded map $\hat{\phi}:\MA(S)\to Y$ such that $F_{\AWSetB}^{\WSetB}\left(\hat{\phi}\right)=\phi$.  Moreover,
\[
\cbnd\left(\hat{\phi}\right)=\bnd(\phi).
\]
\end{thm}

The proof of both theorems is nearly identical to the proof in \cite[Exercise 14.1]{paulsen}.

\subsection{Appending a Zero-Weight Element}\label{zero-weight}

As with weighted sets, an array-weighted set need not have an element of weight 0.  In the weighted set case, one need only append a new element and extend the weight function for the new element to have weight 0 as in \cite[p.\ 7]{grandis2004}.  However, an array-weighted set has a net of weights that must be extended while preserving the existing relations between them.  Since this construction is the prototype for the array-weight structure of a disjoint union, appending a zero-weight element will be shown in detail.

\begin{defn}[Minimally appending a zero-weight element]
Given an array-weighted set $X$, let $Z(X):=X\uplus\{\Theta\}$, the disjoint union of $X$ with a distinguished singleton $\Theta$, which will be the zero-weight element.  For $m,n\in\mathbb{N}$, define $w_{Z(X),m,n}:M_{m,n}(Z(X))\to[-\infty,\infty]$ by the supremum below.
\[
w_{Z(X),m,n}(A):=\sup\left(\left\{w_{X,j,k}\left(A\circ(\sigma\times\tau)\right):\begin{array}{l}\xymatrix{[j]\ar[r]^{\sigma} & [m]}\\ \textrm{ and }\\\xymatrix{[k]\ar[r]^{\tau} & [n]}\\\textrm{one-to-one,}\\A\circ(\sigma\times\tau)\in M_{j,k}(X)\end{array}\right\}\cup\{0\}\right)
\]
\end{defn}

\begin{lem}
Equipped with the above functions, $Z(X)$ is an array-weighted set such that $w_{Z(X),m,n}(A)=w_{X,m,n}(A)$ for all $m,n\in\mathbb{N}$ and $A\in M_{m,n}(X)$.
\end{lem}

\begin{proof}

Fix $m,n\in\mathbb{N}$ and $A\in M_{m,n}(Z(X))$.  From definition, $w_{Z(X),m,n}(A)\geq 0$, so next, the weight of $A$ is shown to be finite.  Given integers $1\leq j\leq m$ and $1\leq k\leq n$, let $\sigma:[j]\to[m]$ and $\tau:[k]\to[n]$ be one-to-one such that $A\circ(\sigma\times\tau)\in M_{j,k}(X)$.  Using the maximality of $\MA$,
\[\begin{array}{rcl}
w_{X,j,k}\left(A\circ(\sigma\times\tau)\right)
&	\leq	&	w_{\MA\left(F_{\AWSetB}^{\WSetB}(X)\right),j,k}\left(A\circ(\sigma\times\tau)\right)\\[12pt]
&	=	&	\sum_{a=1}^{j}\sum_{b=1}^{k}w_{F_{\AWSetB}^{\WSetB}(X)}\left(A(\sigma(a),\tau(b))\right)\\[12pt]
&	\leq	&	\sum_{A(a,b)\in X}w_{F_{\AWSetB}^{\WSetB}(X)}\left(A(a,b)\right),\\[12pt]
\end{array}\]
so a supremum yields
\[
w_{Z(X),m,n}(A)\leq\sum_{A(a,b)\in X}w_{F_{\AWSetB}^{\WSetB}(X)}\left(A(a,b)\right)<\infty.
\]

Next, the weight of $A$ is shown to bound the weights of its subarrays.  For integers $1\leq j\leq m$ and $1\leq k\leq n$, let $\alpha:[j]\to[m]$ and $\beta:[k]\to[n]$ be one-to-one.  For integers $1\leq a\leq j$ and $1\leq b\leq k$, consider one-to-one functions $\sigma:[a]\to[j]$ and $\tau:[b]\to[k]$ such that $A\circ(\alpha\times\beta)\circ(\sigma\times\tau)\in M_{a,b}(X)$.  Note that $(\alpha\times\beta)\circ(\sigma\times\tau)=(\alpha\circ\sigma)\times(\beta\circ\tau)$ is one-to-one as well, so
\[
w_{X,a,b}\left(A\circ(\alpha\times\beta)\circ(\sigma\times\tau)\right)\leq w_{Z(X),m,n}(A).
\]
A supremum then shows that
\[
w_{Z(X),j,k}\left(A\circ(\alpha\times\beta)\right)\leq w_{Z(X),m,n}(A).
\]

Finally, the weight of $A$ is shown to be bounded by the sum of its blocks.  For an integer $1\leq j<m$, let $\alpha:[j]\to[m]$ and $\gamma:[m-j]\to[m]$ be one-to-one such that $\Ran(\alpha)\cap\Ran(\gamma)=\emptyset$.  For integers $1\leq a\leq m$ and $1\leq b\leq n$, consider one-to-one functions $\sigma:[a]\to[m]$ and $\tau:[b]\to[n]$ such that $A\circ(\sigma\times\tau)\in M_{a,b}(X)$.  The range of $\sigma$ could easily intersect with the ranges of both $\alpha$ and $\gamma$, so this entanglement will be handled by constructing two auxiliary functions.

In the case that $\sigma^{-1}(\Ran(\alpha))$ and $\sigma^{-1}(\Ran(\gamma))$ are both nonempty, enumerate each as $\left(s_{l}\right)_{l=1}^{c}$ and $\left(t_{p}\right)_{p=1}^{a-c}$, respectively.  Define $\hat{\alpha}:[c]\to[a]$ by $\hat{\alpha}(l):=s_{l}$ and $\hat{\gamma}:[a-c]\to[a]$ by $\hat{\gamma}(p):=t_{p}$.  Likewise, define $\sigma_{1}:[c]\to[j]$ by $\sigma_{1}(l):=\alpha^{-1}\left(\sigma\left(\hat{\alpha}(l)\right)\right)$ and $\sigma_{2}:[a-c]\to[m-j]$ by $\sigma_{2}(p):=\gamma^{-1}\left(\sigma\left(\hat{\gamma}(p)\right)\right)$.  By design, the following two squares commute in $\Set$.
\[\begin{array}{cc}
\xymatrix{
[c]\ar[r]^{\hat{\alpha}}	\ar[d]_{\sigma_{1}}&	[a]\ar[d]^{\sigma}\\
[j]\ar[r]_{\alpha}	&	[m]\\
}
&
\xymatrix{
[a-c]\ar[r]^(0.6){\hat{\gamma}}\ar[d]_{\sigma_{2}}	&	[a]\ar[d]^{\sigma}\\
[m-j]\ar[r]_(0.6){\gamma}	&	[m]\\
}
\end{array}\]
Notice also that $\Ran\left(\hat{\alpha}\right)\cap\Ran\left(\hat{\gamma}\right)=\emptyset$.  Then,
\[\begin{array}{rcl}
w_{X,a,b}\left(A\circ(\sigma\times\tau)\right)

&	\leq	&	w_{X,c,b}\left(A\circ(\sigma\times\tau)\circ\left(\hat{\alpha}\times id_{[b]}\right)\right)\\[6pt]
&	&	+w_{X,a-c,b}\left(A\circ(\sigma\times\tau)\circ\left(\hat{\gamma}\times id_{[b]}\right)\right)\\[12pt]

&	=	&	w_{X,c,b}\left(A\circ\left(\left(\sigma\circ\hat{\alpha}\right)\times\left(\tau\circ id_{[b]}\right)\right)\right)\\[6pt]
&	&	+w_{X,a-c,b}\left(A\circ\left(\left(\sigma\circ\hat{\gamma}\right)\times\left(\tau\circ id_{[b]}\right)\right)\right)\\[12pt]

&	=	&	w_{X,c,b}\left(A\circ\left(\left(\alpha\circ\sigma_{1}\right)\times\left(id_{[n]}\circ\tau\right)\right)\right)\\[6pt]
&	&	+w_{X,a-c,b}\left(A\circ\left(\left(\gamma\circ\sigma_{2}\right)\times\left(id_{[n]}\circ\tau\right)\right)\right)\\[12pt]

&	=	&	w_{X,c,b}\left(A\circ(\alpha\times id_{[n]})\circ\left(\sigma_{1}\times\tau\right)\right)\\[6pt]
&	&	+w_{X,a-c,b}\left(A\circ(\gamma\times id_{[n]})\circ\left(\sigma_{2}\times\tau\right)\right)\\[12pt]

&	\leq	&	w_{Z(X),j,n}\left(A\circ(\alpha\times id_{[n]})\right)\\[6pt]
&	&	+w_{Z(X),m-j,n}\left(A\circ(\gamma\times id_{[n]})\right).\\[12pt]
\end{array}\]

In the case $\Ran(\sigma)\subseteq\Ran(\alpha)$, then $a\leq j$.  Define $\hat{\sigma}:[a]\to[j]$ by $\hat{\sigma}(l):=\alpha^{-1}(\sigma(l))$, which is one-to-one.  By design $\alpha\circ\hat{\sigma}=\sigma$.  Then,
\[\begin{array}{rcl}
w_{X,a,b}\left(A\circ(\sigma\times\tau)\right)
&	=	&	w_{X,a,b}\left(A\circ\left(\left(\alpha\circ\hat{\sigma}\right)\times\left(id_{[n]}\circ\tau\right)\right)\right)\\[6pt]
&	=	&	w_{X,a,b}\left(A\circ\left(\alpha\times id_{[n]}\right)\circ\left(\hat{\sigma}\times\tau\right)\right)\\[6pt]
&	\leq	&	w_{Z(X),j,n}\left(A\circ\left(\alpha\times id_{[n]}\right)\right)\\[6pt]
&	\leq	&	w_{Z(X),j,n}\left(A\circ\left(\alpha\times id_{[n]}\right)\right)\\[6pt]
&	&	+w_{Z(X),m-j,n}\left(A\circ(\gamma\times id_{[n]})\right).\\[12pt]
\end{array}\]
A similar calculation occurs in the case $\Ran(\sigma)\subseteq\Ran(\gamma)$.

Taking all three cases into account, a supremum then gives
\[
w_{Z(X),m,n}\left(A\right)\leq w_{Z(X),j,n}\left(A\circ(\alpha\times id_{[n]})\right)+w_{Z(X),m-j,n}\left(A\circ(\gamma\times id_{[n]})\right)
\]
as desired.  A similar argument shows the same result in the second coordinate.

In the case $A\in M_{m,n}(X)$,
\[
w_{X,j,k}(A\circ(\sigma\times\tau))\leq w_{X,m,n}(A)=w_{X,m,n}\left(A\circ\left(id_{[m]}\times id_{[n]}\right)\right)
\]
for all integers $1\leq j\leq m$, $1\leq k\leq n$, $A\in M_{m,n}(X)$, and one-to-one functions $\sigma:[j]\to[m]$ and $\tau:[k]\to[n]$.  Thus, $w_{Z(X),m,n}(A)=w_{X,m,n}(A)$.

\end{proof}

Note that $w_{Z(X),1,1}(\Theta)=0$, and moreover, this is the least such array-weight.

\begin{thm}[Minimality of $Z(X)$]\label{minimality-zx}
Let $\left(v_{m,n}\right)_{m,n\in\mathbb{N}}$ be another array-weight on $Z(X)$ such that $v_{m,n}(A)=w_{X,m,n}(A)$ for all $m,n\in\mathbb{N}$ and $A\in M_{m,n}(X)$.  Then, $v_{m,n}(A)\geq w_{Z(X),m,n}(A)$ for all $m,n\in\mathbb{N}$ and $A\in M_{m,n}(Z(X))$.
\end{thm}

\begin{proof}

For integers $1\leq j\leq m$ and $1\leq k\leq n$, consider one-to-one functions $\sigma:[j]\to[m]$ and $\tau:[k]\to[n]$ such that $A\circ(\sigma\times\tau)\in M_{m,n}(X)$.  Then,
\[
w_{X,j,k}\left(A\circ(\sigma\times\tau)\right)
=v_{j,k}\left(A\circ(\sigma\times\tau)\right)
\leq v_{m,n}(A).
\]
A supremum gives the result.

\end{proof}

\subsection{The Disjoint Union of Array-Weighted Sets}\label{set-constructions}

As done with $\WSetC$ in \cite[p.\ 7]{grandis2004} and $\MBanB$ in \cite[p.\ 247]{effros1988}, subsets and quotients of array-weighted sets inherit a natural array-weight structure and characterize the equalizer and coequalizer, respectively, in $\AWSetB$.  Products are formed by a restricted cartesian product equipped with an $\ell^{\infty}$-type structure on each matrix level, much like $\BanC$ in \cite[Example 2.1.7.d]{borceux1} and $\MBanC$ by extension.

The coproduct, however, is more difficult to describe.  Like $\WSetC$ in \cite[p.\ 7]{grandis2004}, the underlying set is a disjoint union, but extending the array-weights is nontrivial as shown in Section \ref{zero-weight}.  Since this construction will be useful in the examples of Sections \ref{constructions} and \ref{mbanachalgebra}, the array-weight structure on the disjoint union will be shown in detail.

\begin{defn}[Array-weight on a disjoint union]
Given an index set $\Lambda$, let $\left(X_{\lambda}\right)_{\lambda\in\Lambda}$ be array-weighted sets.  Let
\[
D:=\bigcup_{\lambda\in\Lambda}\left(\{\lambda\}\times X_{\lambda}\right)
\]
be the disjoint union of the underlying sets with canonical inclusions $\rho_{\lambda}:X_{\lambda}\to D$ by $\rho_{\lambda}(x):=(\lambda,x)$.  For $m,n\in\mathbb{N}$, define $w_{D,m,n}:M_{m,n}(D)\to[-\infty,\infty]$ by
\[
w_{D,m,n}(A):=\sup\left\{w_{Y,m,n}\left(M_{m,n}(\phi)(A)\right):\begin{array}{c}Y\textrm{ an array-weighted set},\\ \xymatrix{D\ar[r]^{\phi} & Y}\textrm{ a function},\\ \phi\circ\rho_{\lambda}\textrm{ completely contractive}\\ \forall\lambda\in\Lambda\end{array}\right\}.
\]
\end{defn}

\begin{lem}\label{array-weight-coproduct}
Equipped with the above functions, $D$ is an array-weighted set.
\end{lem}

\begin{proof}

Fix $m,n\in\mathbb{N}$ and $A\in M_{m,n}(D)$.  First, the supremum is shown to be of a nonempty set.  Let $c_{0}:D\to \mathbb{C}$ be the constant map to 0.  Trivially, $c_{0}\circ\rho_{\lambda}$ is completely contractive to $\MIN(\mathbb{C})$ for all $\lambda\in\Lambda$, so
\[
w_{D,m,n}(A)\geq w_{\MIN(\mathbb{C}),m,n}\left(M_{m,n}\left(c_{0}\right)(A)\right)=0.
\]

Next, the supremum is shown to be finite.  For each $1\leq j\leq m$ and $1\leq k\leq n$, there is $\lambda_{j,k}\in\Lambda$ and $x_{j,k}\in X_{\lambda_{j,k}}$ such that $A(j,k)=\rho_{\lambda_{j,k}}\left(x_{j,k}\right)$.  Let $Y$ be any array-weighted set and $\phi:D\to Y$ a function such that $\phi\circ\rho_{\lambda}$ is completely contractive for all $\lambda\in\Lambda$.  Using the maximality of $\MA$,
\[\begin{array}{rcl}
w_{Y,m,n}\left(M_{m,n}(\phi)(A)\right)
&	\leq	&	w_{\MA\left(F_{\AWSetB}^{\WSetB}(Y)\right),m,n}\left(M_{m,n}(\phi)(A)\right)\\[12pt]
&	=	&	\sum_{j=1}^{m}\sum_{k=1}^{n}w_{F_{\AWSetB}^{\WSetB}(Y)}\left(\phi(A(j,k))\right)\\[12pt]
&	=	&	\sum_{j=1}^{m}\sum_{k=1}^{n}w_{Y,1,1}\left(\phi\left(\rho_{\lambda_{j,k}}\left(x_{j,k}\right)\right)\right)\\[12pt]
&	\leq	&	\sum_{j=1}^{m}\sum_{k=1}^{n}w_{X_{\lambda_{j,k}},1,1}\left(x_{j,k}\right)\\[12pt]
\end{array}\]
so a supremum yields
\[
w_{D,m,n}(A)\leq\sum_{j=1}^{m}\sum_{k=1}^{n}w_{X_{\lambda_{j,k}},1,1}\left(x_{j,k}\right)<\infty.
\]

Now, the weight of $A$ is shown to bound the weights of its subarrays.  Let $\alpha:[j]\to[m]$ and $\beta:[k]\to[n]$ be one-to-one.  A quick check shows that
\[
M_{j,k}(\phi)\left(A\circ(\alpha\times\beta)\right)=\left(M_{m,n}(\phi)(A)\right)\circ(\alpha\times\beta),
\]
so
\[\begin{array}{rcl}
w_{Y,j,k}\left(M_{j,k}(\phi)\left(A\circ(\alpha\times\beta)\right)\right)
&	=	&	w_{Y,j,k}\left(\left(M_{m,n}(\phi)(A)\right)\circ(\alpha\times\beta)\right)\\[10pt]
&	\leq	&	w_{Y,m,n}\left(M_{m,n}(\phi)(A)\right)\\[10pt]
&	\leq	&	w_{D,m,n}(A).\\[10pt]
\end{array}\]
A supremum then shows that
\[
w_{D,j,k}\left(A\circ(\alpha\times\beta)\right)\leq w_{D,m,n}(A).
\]

Finally, the weight of $A$ is shown to be bounded by the sum of its blocks.  Let $\alpha:[j]\to[m]$ and $\gamma:[m-j]\to[m]$ be one-to-one such that ${\Ran(\alpha)\cap\Ran(\gamma)=\emptyset}$.  Then,
\[\begin{array}{rcl}
w_{Y,m,n}\left(M_{m,n}(\phi)(A)\right)

&	\leq	&	w_{Y,j,n}\left(\left(M_{m,n}(\phi)(A)\right)\circ\left(\alpha\times id_{[n]}\right)\right)\\[8pt]
&	&	+w_{Y,m-j,n}\left(\left(M_{m,n}(\phi)(A)\right)\circ\left(\gamma\times id_{[n]}\right)\right)\\[10pt]

&	=	&	w_{Y,j,n}\left(M_{j,n}(\phi)\left(A\circ\left(\alpha\times id_{[n]}\right)\right)\right)\\[8pt]
&	&	+w_{Y,m-j,n}\left(M_{m-j,n}(\phi)\left(A\circ\left(\gamma\times id_{[n]}\right)\right)\right)\\[10pt]

&	\leq	&	w_{D,j,n}\left(A\circ\left(\alpha\times id_{[n]}\right)\right)\\[8pt]
&	&	+w_{D,m-j,n}\left(A\circ\left(\gamma\times id_{[n]}\right)\right)\\[10pt]
\end{array}\]
A supremum then gives
\[
w_{D,m,n}\left(A\right)\leq w_{D,j,n}\left(A\circ(\alpha\times id_{[n]})\right)+w_{D,m-j,n}\left(A\circ(\gamma\times id_{[n]})\right)
\]
as desired.  A similar argument shows same result in the second coordinate.

\end{proof}

Notice that the inclusion maps are completely isometric as shown by appending a zero-weight element.

\begin{lem}
For $\lambda\in\Lambda$, $m,n\in\mathbb{N}$, and $A\in M_{m,n}\left(X_{\lambda}\right)$, then
\[
w_{D,m,n}\left(M_{m,n}\left(\rho_{\lambda}\right)(A)\right)=w_{X_{\lambda},m,n}(A).
\]
\end{lem}

\begin{proof}

Define $\phi:D\to Z\left(X_{\lambda}\right)$ by
\[
\phi(\mu,x):=\left\{\begin{array}{cc}
x,	&	\mu=\lambda,\\
\Theta,	&	\mu\neq\lambda.
\end{array}\right.
\]
A quick calculation shows that $\phi\circ\rho_{\mu}$ is completely contractive for all $\mu\in\Lambda$.  Thus,
\[\begin{array}{rcl}
w_{X_{\lambda},m,n}(A)
&	=	&	w_{Z\left(X_{\lambda}\right),m,n}(A)\\[10pt]
&	=	&	w_{Z\left(X_{\lambda}\right),m,n}\left(M_{m,n}(\phi)\left(M_{m,n}\left(\rho_{\lambda}\right)\left(A\right)\right)\right)\\[10pt]
&	\leq	&	w_{D,m,n}\left(M_{m,n}\left(\rho_{\lambda}\right)\left(A\right)\right).\\[10pt]
\end{array}\]
For any array-weighted set $Y$ and function $\phi:D\to Y$ such that $\phi\circ\rho_{\mu}$ is completely contractive for all $\mu\in\Lambda$,
\[\begin{array}{rcl}
w_{Y,m,n}\left(M_{m,n}(\phi)\left(M_{m,n}\left(\rho_{\lambda}\right)\left(A\right)\right)\right)
&	=	&	w_{Y,m,n}\left(M_{m,n}\left(\phi\circ\rho_{\lambda}\right)\left(A\right)\right)\\[10pt]
&	\leq	&	w_{X_{\lambda},m,n}(A).\\[10pt]
\end{array}\]
A supremum then gives
\[
w_{D,m,n}\left(M_{m,n}\left(\rho_{\lambda}\right)\left(A\right)\right)\leq w_{X_{\lambda},m,n}(A).
\]

\end{proof}

So constructed, $D$ has the following universal property, analogous to the matricial $\ell^{1}$-direct sum of Theorem \ref{MBan-coproduct}.

\begin{thm}[Universal property of the coproduct, $\AWSetB$]\label{awsetb-coproduct}
For an array-weighted set $Y$, let $\phi_{\lambda}:X_{\lambda}\to Y$ be completely bounded maps satisfying
\[
\sup\left\{\cbnd\left(\phi_{\lambda}\right):\lambda\in\Lambda\right\}<\infty.
\]
There is a unique completely bounded map $\phi:D\to Y$ such that $\phi\circ\rho_{\lambda}=\phi_{\lambda}$ for all $\lambda\in\Lambda$.  Moreover,
\[
\cbnd(\phi)=\sup\left\{\cbnd\left(\phi_{\lambda}\right):\lambda\in\Lambda\right\}.
\]
\end{thm}

\begin{proof}

Let $L$ be the supremum above.  Define $\phi:D\to Y$ by $\phi(\lambda,x):=\phi_{\lambda}(x)$, the coproduct map in $\Set$.  By design, $\phi\circ\rho_{\lambda}=\phi_{\lambda}$, and uniqueness follows from the universal property of the coproduct in $\Set$.  All that remains is to prove that $\phi$ is completely bounded and that $\cbnd(\phi)=L$.

To that end, if $L=0$, then $\cbnd\left(\phi_{\lambda}\right)=0$ for all $\lambda\in\Lambda$.  Consequently,
\[
w_{Y,1,1}\left(\phi(\lambda,x)\right)
=w_{Y,1,1}\left(\phi_{\lambda}(x)\right)
=0
\]
for all $\lambda\in\Lambda$ and $x\in X_{\lambda}$.  Using $\MA$, $w_{Y,m,n}\left(M_{m,n}(\phi)(A)\right)=0$ for all $m,n\in\mathbb{N}$ and $A\in M_{m,n}(D)$.  Thus, $\phi$ is trivially completely bounded, and $\cbnd(\phi)=0=L$.

If $L\neq 0$, let $Y_{L}$ be the underlying set of $Y$ equipped with the scaled array-weight $w_{Y_{L},m,n}:=\frac{1}{L}w_{Y,m,n}$.  Observe that the identity map $\sigma:Y_{L}\to Y$ is an isomorphism in $\AWSetB$ with $\cbnd(\sigma)=L$ and $\cbnd\left(\sigma^{-1}\right)=\frac{1}{L}$.  For all $\lambda\in\Lambda$,
\[
\cbnd\left(\sigma^{-1}\circ\phi\circ\rho_{\lambda}\right)
=\cbnd\left(\sigma^{-1}\circ\phi_{\lambda}\right)
\leq\cbnd\left(\sigma^{-1}\right)\cbnd\left(\phi_{\lambda}\right)
\leq1.
\]
Consequently,
\[
\frac{1}{L}w_{Y,m,n}\left(M_{m,n}(\phi)(A)\right)
=w_{Y_{L},m,n}\left(M_{m,n}(\phi)(A)\right)
\leq w_{D,m,n}(A),
\]
or rather, $w_{Y,m,n}\left(M_{m,n}(\phi)(A)\right)\leq Lw_{D,m,n}(A)$.  Hence, $\phi$ is completely bounded and $\cbnd(\phi)\leq L$.  Equality can be shown using arrays from each $X_{\lambda}$.

\end{proof}

Letting $\AWSetC$ be the category of array-weighted sets with completely contractive maps, this theorem guarantees that $\AWSetC$ has all coproducts.  As such, the notation
\[
{\coprod_{\lambda\in\Lambda}}^{\AWSetC}X_{\lambda}
\]
will be used to denote the disjoint union of the family $\left(X_{\lambda}\right)_{\lambda\in\Lambda}$.

Admittedly, the description of the array-weight for the disjoint union is not ideal as it relies upon an abstract supremum.  One would like to have a more intrinsic or explicit description of the array-weight, but this structure remains nebulous in general.  Even when the constituent sets arise as subsets of a common array-weighted set, the resultant array-weight on the disjoint union may not be immediately obvious.

\begin{ex}\label{badness2}
Consider $X:=\{1,-1\}$ as an array-weighted subset of $\MIN(\mathbb{C})$, and let $X_{1}:=\{1\}$ and $X_{2}:=\{-1\}$ be considered as array-weighted subsets of $X$.  Note that
\[
w_{X,2,2}\left(\begin{bmatrix}1 & 1\\ 1 & -1\\\end{bmatrix}\right)=\sqrt{2}
\]
while
\[
w_{\AMAX(\mathbb{C}),2,2}\left(\begin{bmatrix}1 & 1\\ 1 & -1\\\end{bmatrix}\right)=2\sqrt{2}
\]
and
\[
w_{\MA(\mathbb{C}),2,2}\left(\begin{bmatrix}1 & 1\\ 1 & -1\\\end{bmatrix}\right)=4.
\]
Notice that the map $\phi:D\to\AMAX(\mathbb{C})$ by $\phi(x):=x$ satisfies $\phi\circ\rho_{n}$ is completely contractive for $n=1,2$.  Hence,
\[
w_{X,2,2}\left(\begin{bmatrix}1 & 1\\ 1 & -1\\\end{bmatrix}\right)
<2\sqrt{2}
\leq w_{D,2,2}\left(\begin{bmatrix}1 & 1\\ 1 & -1\\\end{bmatrix}\right)
\leq 4.
\]
The above inequality shows that $X\not\cong_{\AWSetC}D$.  Hence, even though $X_{1}$ and $X_{2}$ inherited their array-weight structure from $X$ and together constitute $X$, their disjoint union structure is distinct from $X$.
\end{ex}

The behavior shown in the example above demonstrates that $\AWSetC$ does not behave quite the same as $\WSetC$.  In $\WSetC$, every object can be decomposed into a coproduct in a natural way.  The following notation was a suggestion from a referee of \cite{grilliette1}.

\begin{defn}[The constantly weighted set]
Given a set $S$ and $\lambda\in[0,\infty)$, let $W_{\lambda}(S)$ denote the set $S$ equipped with the constant weight function $w_{W_{\lambda}(S)}:S\to[0,\infty)$ by $w_{W_{\lambda}(S)}(s):=\lambda$.
\end{defn}

\begin{prop}[Decomposition of a weighted set]\label{decomp-wset}
Given a weighted set $S$,
\[
S\cong_{\WSetC}{\coprod_{s\in S}}^{\WSetC}W_{w_{S}(s)}\left(\{s\}\right),
\]
\end{prop}

The proof of the proposition is immediate from direct calculation.  However, Example \ref{badness2} shows that such a decomposition is not always possible for an arbitrary array-weighted set.  However, if an array-weighted set has the maximum structure, such a decomposition is immediate.

\begin{cor}[Decomposition of a maximally array-weighted set]
Given a weighted set $S$,
\[
\MA(S)\cong_{\AWSetC}{\coprod_{s\in S}}^{\AWSetC}\MA\left(W_{w_{S}(s)}\left(\{s\}\right)\right).
\]
\end{cor}

The proof follows as $\MA$ is a left adjoint functor applied to a coproduct.  Moreover, the disjoint union of array-weighted sets gives another method of appending a zero-weight element, as illustrated in the following example.

\begin{ex}\label{badness1}
Consider $X:=\{1,0\}$ as an array-weighted subset of $\MIN(\mathbb{C})$, and let $X_{1}:=\{1\}$ and $X_{2}:=\{0\}$ be considered as array-weighted subsets of $X$.  Note that
\[
w_{Z\left(X_{1}\right),2,2}\left(\begin{bmatrix}1 & 1\\ 1 & \Theta\\\end{bmatrix}\right)=\sqrt{2}
\]
while
\[
w_{X,2,2}\left(\begin{bmatrix}1 & 1\\ 1 & 0\\\end{bmatrix}\right)=\frac{1}{2}+\frac{1}{2}\sqrt{5}.
\]
and
\[
w_{\MA\left(\mathbb{C}\right),2,2}\left(\begin{bmatrix}1 & 1\\ 1 & 0\\\end{bmatrix}\right)=3.
\]
However, note also that
\[
w_{\AMAX(\mathbb{C}),2,2}\left(\begin{bmatrix}1 & 1\\ 1 & 0\\\end{bmatrix}\right)=\sqrt{5},
\]
and that the map $\phi:D\to\AMAX(\mathbb{C})$ by $\phi(x):=x$ satisfies $\phi\circ\rho_{n}$ is completely contractive for $n=1,2$.  Hence,
\[
w_{Z\left(X_{1}\right),2,2}\left(\begin{bmatrix}1 & 1\\ 1 & \Theta\\\end{bmatrix}\right)
<w_{X,2,2}\left(\begin{bmatrix}1 & 1\\ 1 & 0\\\end{bmatrix}\right)
<\sqrt{5}
\leq w_{D,2,2}\left(\begin{bmatrix}1 & 1\\ 1 & 0\\\end{bmatrix}\right)
\leq3.
\]
Notably, the inequality above shows that $Z\left(X_{1}\right)\not\cong_{\AWSetC}D$, and both are distinct from $X$.
\end{ex}

Recall that Theorem \ref{minimality-zx} describes the least way to append a zero-weight element.  The above example seems to imply that the disjoint union structure would be the greatest way to append a zero-element, which is indeed the case.

\begin{cor}
Let $\mathfrak{Z}:=\{\Theta\}$ be equipped with the trivial array-weight
\[
w_{\mathfrak{Z},m,n}(B):=0
\]
for all $m,n\in\mathbb{N}$ and $B\in M_{m,n}(\mathfrak{Z})$.  Given an array-weighted set $X$, let $\left(v_{m,n}\right)_{m,n\in\mathbb{N}}$ be another array-weight on $Z(X)$ such that $v_{1,1}(\Theta)=0$ and $v_{m,n}(C)=w_{X,m,n}(C)$ for all $m,n\in\mathbb{N}$ and $C\in M_{m,n}(X)$.  Then,
\[
v_{m,n}(A)\leq w_{X{\coprod}^{\AWSetC}\mathfrak{Z},m,n}(A)
\]
for all $m,n\in\mathbb{N}$ and $A\in M_{m,n}(Z(X))$.
\end{cor}

\begin{proof}

Let $Y$ denote $Z(X)$ equipped with the array-weight $\nu$.  Define $\phi_{1}:X\to Y$ by $\phi_{1}(x):=x$ and $\phi_{2}:\mathfrak{Z}\to Y$ by $\phi_{2}(\Theta):=\Theta$, the usual set inclusions.  By assumption, $\phi_{1}$ is completely isometric.  For $m,n\in\mathbb{N}$ and $B\in M_{m,n}(\mathfrak{Z})$, observe that
\[\begin{array}{rcl}
\nu_{m,n}(B)
&	\leq	&	w_{\MA\left(F_{\AWSetB}^{\WSetB}(Y)\right),m,n}(B)\\[10pt]
&	=	&	\sum_{j=1}^{m}\sum_{k=1}^{n}w_{F_{\AWSetB}^{\WSetB}(Y)}(B(j,k))\\[10pt]
&	=	&	\sum_{j=1}^{m}\sum_{k=1}^{n}w_{F_{\AWSetB}^{\WSetB}(Y)}(\Theta)\\[10pt]
&	=	&	0\\[10pt]
&	=	&	w_{\mathfrak{Z},m,n}(B),\\[10pt]
\end{array}\]
meaning that $\phi_{2}$ is also completely isometric.  By Theorem \ref{awsetb-coproduct}, there is a unique completely contractive map $\phi:X{\coprod}^{\AWSetC}\mathfrak{Z}\to Y$ such that $\phi\circ\rho_{j}=\phi_{j}$ for all $j=1,2$.  For $m,n\in\mathbb{N}$ and $A\in M_{m,n}(Z(X))$, a calculation shows
\[
\nu_{m,n}(A)
=\nu_{m,n}\left(M_{m,n}(\phi)(A)\right)
\leq w_{X{\coprod}^{\AWSetC}\mathfrak{Z},m,n}(A).
\]

\end{proof}

\subsection{Array-Free Elements}\label{array-free}

In pure algebra, the basis of a vector space is traditionally shown to be linearly independent by using characteristic functions, regarded as functions into the field.  This relationship between characteristic functions and linear independence motivates the following definition.

\begin{defn}
For an array-weighted set $X$, an element $x\in X$ is \emph{array-free} in $X$ if the characteristic function of $x$ is completely bounded when regarded as a map from $X$ to $\MIN(\mathbb{C})$.
\end{defn}

This first example illuminates the relationship between array-freeness and linear independence.

\begin{ex}
Let $V$ be a matrix-normed space and $X\subset V$ a finite, linearly independent subset equipped with the inherited array-weight from $V$.  For $x\in X$, consider the characteristic function $\chi:X\to\mathbb{C}$ of $x$.  Letting $W:=\Span(X)$, define a linear map $\phi:W\to\mathbb{C}$ on the basis $X$ by $\phi(y):=\chi(y)$ for all $y\in X$.  As $W$ is finite-dimensional, $\phi$ is bounded.  Letting $\iota:W\to V$ be the inclusion map, there is a bounded linear map $\varphi:V\to\mathbb{C}$ such that $\varphi\circ\iota=\phi$ by the Hahn-Banach Theorem.  By Theorem \ref{MIN-prop}, $\varphi$ is completely bounded from $V$ to $\MIN(\mathbb{C})$.  Letting $\epsilon:X\to W$ be the inclusion of generators, then $\varphi\circ\iota\circ\epsilon=\chi$ is completely bounded.  Thus, $x$ is array-free in $X$.
\end{ex}

The next proposition motivates the nomenclature.

\begin{prop}[Array-free and finite support functions]
Given an array-weighted set $X$, let $Y$ be a subset of $X$.  All functions from $X$ to $\MIN(\mathbb{C})$ with finite support contained in $Y$ are completely bounded if and only if $x$ is array-free in $X$ for all $x\in Y$.
\end{prop}

\begin{proof}

$(\Rightarrow)$ This direction is immediate as a characteristic function has finite support.

$(\Leftarrow)$  Let $\phi:X\to\mathbb{C}$ have finite support $\left(x_{j}\right)_{j=1}^{p}\subseteq Y$.  Then, $\phi=\sum_{j=1}^{p}\phi\left(x_{j}\right)\chi_{j}$, where $\chi_{j}$ is the characteristic function of $x_{j}$.  For $m,n\in\mathbb{N}$ and $A\in M_{m,n}(X)$, a quick calculation shows that
\[
M_{m,n}(\phi)(A)=\sum_{j=1}^{p}\phi\left(x_{j}\right)M_{m,n}\left(\chi_{j}\right)(A),
\]
so
\[\begin{array}{rcl}
\left\|M_{m,n}(\phi)(A)\right\|_{\MIN(\mathbb{C}),m,n}
&	\leq	&	\sum_{j=1}^{p}\left|\phi\left(x_{j}\right)\right|\left\|M_{m,n}\left(\chi_{j}\right)(A)\right\|_{\MIN(\mathbb{C}),m,n}\\[10pt]
&	\leq	&	\left(\sum_{j=1}^{p}\left|\phi\left(x_{j}\right)\right|\cbnd\left(\chi_{j}\right)\right)w_{X,m,n}(A).\\[10pt]
\end{array}\]
Thus, $\phi$ is completely bounded.

\end{proof}

Consequently, array-free elements act like free elements in the sense that a finite number of them can be mapped arbitrarily while the remainder of the set is annihilated.  The most important case for Sections \ref{constructions} and \ref{mbanachalgebra} is when all elements of an array-weighted set are array-free.  To detect this quickly, the following metric is introduced.

\begin{defn}
Given an array-weighted set $X$, the \emph{bounded range number} of $X$ is
\[
\brn(X):=\inf\left\{\frac{w_{X,m,n}(A)}{\sqrt{mn}}:m,n\in\mathbb{N},A\in M_{m,n}(X)\right\}.
\]
\end{defn}

The value finds its name in the following theorem.

\begin{thm}[Bounded range maps and $\brn$]
Given an array-weighted set $X$, the following are equivalent:
\begin{enumerate}
\item all bounded range maps from $X$ to $\MIN(\mathbb{C})$ are completely bounded;
\item the constant map to 1 regarded as a map from $X$ to $\MIN(\mathbb{C})$ is completely bounded;
\item $\brn(X)>0$.
\end{enumerate}
In this case, $x$ is array-free in $X$ for all $x\in X$.
\end{thm}

\begin{proof}

$(1\Rightarrow 2)$ This is immediate as the constant map to 1 has bounded range.

$(2\Rightarrow 3)$ Let $c_{1}:X\to\mathbb{C}$ be the constant map to 1.  For all $m,n\in\mathbb{N}$ and $A\in M_{m,n}(X)$,
\[
\sqrt{mn}
=w_{\MIN(\mathbb{C}),m,n}\left(M_{m,n}\left(c_{1}\right)(A)\right)
\leq\cbnd\left(c_{1}\right)w_{X,m,n}(A).
\]
Thus, $\cbnd\left(c_{1}\right)\neq 0$ and
\[
\frac{w_{X,m,n}(A)}{\sqrt{mn}}\geq\frac{1}{\cbnd\left(c_{1}\right)},
\]
so $\brn(X)\geq\frac{1}{\cbnd\left(c_{1}\right)}>0$.

$(3\Rightarrow 1)$ Let $\phi:X\to\mathbb{C}$ have range bounded by $M$.  For $m,n\in\mathbb{N}$ and $A\in M_{m,n}(X)$, then
\[\begin{array}{rcl}
\left\|M_{m,n}(\phi)(A)\right\|_{\MIN(\mathbb{C}),m,n}
&	\leq	&	\left(\sum_{j=1}^{m}\sum_{k=1}^{n}\left|\phi(A(j,k))\right|^{2}\right)^{1/2}\\[10pt]
&	\leq	&	\left(\sum_{j=1}^{m}\sum_{k=1}^{n}M^{2}\right)^{1/2}\\[10pt]
&	=	&	M\sqrt{mn}\\[10pt]
&	\leq	&	\frac{M}{\brn(X)}w_{X,m,n}(A).\\[10pt]
\end{array}\]
Thus, $\phi$ is completely bounded.

\end{proof}

Combining the previous two results gives the following statement for finite array-weighted sets.

\begin{cor}[Array-free and finite sets]\label{finite-array-free}
Given a finite array-weighted set $X$, the following are equivalent:
\begin{enumerate}
\item all maps from $X$ to $\MIN(\mathbb{C})$ are completely bounded;
\item the constant map to 1 regarded as a map from $X$ to $\MIN(\mathbb{C})$ is completely bounded;
\item $\brn(X)>0$;
\item $x$ is array-free in $X$ for all $x\in X$.
\end{enumerate}
\end{cor}

While Criterion (3) is both necessary and sufficient for all elements of a finite set to be array-free, it is not necessary for infinite sets.

\begin{ex}
Let $V$ be $\ell^{\infty}$ with any matrix-norm.  Letting $\left(\vec{e}_{j}\right)_{j\in\mathbb{N}}\subset V$ be the standard basis, define $x_{j}:=\frac{1}{j}\vec{e}_{j}$ and $X:=\left\{x_{j}:j\in\mathbb{N}\right\}\subset V$ with the inherited array-weight from $V$.  Then, $\brn(X)=0$.

For $j\in\mathbb{N}$, consider the characteristic function $\chi:X\to\mathbb{C}$ of $x_{j}$.  Define a bounded linear map $\phi:\ell^{\infty}\to\mathbb{C}$ by $\phi\left(\vec{x}\right):=j\vec{x}(j)$, the scaled evaluation map at $j$.  By Theorem \ref{MIN-prop}, $\phi$ is completely bounded from $V$ to $\MIN(\mathbb{C})$.  Letting $\epsilon:X\to V$ be the inclusion map, then $\phi\circ\epsilon=\chi$ is completely bounded.  Thus, $x_{j}$ is array-free in $X$.
\end{ex}

The maximal and minimal array-weight structures give stark extremes on array-freeness.  The maximal array-weight behaves exactly like a weighted set in this regard.

\begin{ex}
Let $S$ be a weighted set.  For $s\in S$, let $\chi:S\to\mathbb{C}$ be the characteristic function of $s$.  Then, $\chi$ is bounded if and only if $w_{S}(s)\neq 0$.  Consequently, $\chi$ is completely bounded from $\MA(S)$ to $\MIN(\mathbb{C})$ if and only if $w_{S}(s)\neq 0$ by Theorem \ref{MA-prop}.  Therefore, $s$ is array-free in $\MA(S)$ if and only if $w_{S}(s)\neq0$.
\end{ex}

On the other hand, no element from a set with the minimal array-weight is array-free, regardless of the underlying weight function.

\begin{ex}\label{badness3}
Let $S$ be a weighted set and $\phi:\mA(S)\to\MIN(\mathbb{C})$ a completely bounded function.  For any $s\in S$ and $n\in\mathbb{N}$, let $A_{n}$ be the $n\times n$-array with only $s$ as an entry.  Likewise, let $J_{n,n}$ be the $n\times n$-matrix with only 1 as an entry.  Then, $M_{n,n}(\phi)\left(A_{n}\right)=\phi(s)J_{n,n}$, so
\[
|\phi(s)|n
=\left\|M_{n,n}(\phi)\left(A_{n}\right)\right\|_{\MIN(\mathbb{C}),n,n}
\leq\cbnd(\phi)w_{\mA(S),n,n}\left(A_{n}\right)
=\cbnd(\phi)w_{S}(s).
\]
Hence, $\phi(s)=0$, meaning that $\phi$ cannot be the characteristic function of $s$.  Therefore, $s$ is not array-free in $\mA(S)$.
\end{ex}

In the previous two examples, elements which were not array-free were automatically mapped to 0.  However, this need not be the case.

\begin{ex}\label{badness4}
Consider $X:=\left\{1,-1\right\}$ with the inherited array-weight from $\MIN(\mathbb{C})$.  Observe that the natural inclusion map $\iota:X\to\MIN(\mathbb{C})$ is completely isometric, and neither element is mapped to 0.  Unfortunately, neither element is array-free.  To show this fact, $\brn(X)$ will be shown to be 0.

To that end, let $J_{m,n}$ be the $m\times n$-matrix with all entries 1.  Inductively construct the following sequence of matrices.
\[\begin{array}{rcl}
A_{0}	&	:=	&	\begin{bmatrix}1\end{bmatrix},\\[15pt]
A_{m+1}	&	:=	&	\begin{bmatrix}J_{1,2^{m}} & -J_{1,2^{m}}\\ A_{m} & A_{m}\end{bmatrix}\forall m\in\mathbb{W},\\[15pt]
\end{array}\]
where $\mathbb{W}:=\mathbb{N}\cup\{0\}$ is the whole numbers.  From definition, $A_{0}\in M_{1,1}(X)$.  For induction, assume for some $m\in\mathbb{W}$ that $A_{m}\in M_{m+1,2^{m}}(X)$.  Then, $A_{m+1}$ has $1+(m+1)=m+2$ rows and $2\cdot 2^{m}=2^{m+1}$ columns.  Moreover, all entries in $A_{m+1}$ are either from $A_{m}$, $J_{1,2^{m}}$, or $-J_{1,2^{m}}$.  Consequently, the entries of $A_{m+1}$ are either $1$ or $-1$.  Thus, $A_{m+1}\in M_{m+2,2^{m+1}}(X)$ as desired.

Notice that $A_{0}A_{0}^{*}=I_{1}$ and
\[
A_{m+1}A_{m+1}^{*}=\begin{bmatrix}2^{m+1} & 0\\ 0 & 2A_{m}A_{m}^{*}\end{bmatrix}
\]
for $m\in\mathbb{W}$, where $I_{m}$ is the identity of $\mathbb{M}_{m,m}$.  For induction, assume for some $m\in\mathbb{W}$ that
$A_{m}A_{m}^{*}=2^{m}I_{m+1}$.
Then,
\[\begin{array}{rcl}
A_{m+1}A_{m+1}^{*}
&	=	&	\begin{bmatrix}2^{m+1} & 0\\ 0 & 2\cdot2^{m}I_{m+1}\end{bmatrix}\\[20pt]
&	=	&	\begin{bmatrix}2^{m+1} & 0\\ 0 & 2^{m+1}I_{m+1}\end{bmatrix}\\[20pt]
&	=	&	2^{m+1}\begin{bmatrix}1 & 0\\ 0 & I_{m+1}\end{bmatrix}\\[20pt]
&	=	&	2^{m+1}I_{m+2}\\
\end{array}\]
as desired.

Computing directly, $w_{X,m+1,2^{m}}\left(A_{m}\right)=2^{m/2}$ and
\[
\brn(X)
\leq\frac{w_{X,m+1,2^{m}}\left(A_{m}\right)}{\sqrt{(m+1)2^{m}}}
=\frac{2^{m/2}}{2^{m/2}\sqrt{m+1}}
=\frac{1}{\sqrt{m+1}}
\]
for all $m\in\mathbb{W}$.  Consequently, $\brn(X)=0$.

By Theorem \ref{finite-array-free}, at least one of $1$ or $-1$ is not array-free.  Let $\chi_{x}$ denote the characteristic function of $x\in X$ and observe that $\iota=\chi_{1}-\chi_{-1}$.  Consequently, $\chi_{1}$ is completely bounded if and only if $\chi_{-1}$ is.  Hence, $1$ is array-free if and only if $-1$ is, meaning neither can be.
\end{ex}

\section{Scaled-Free Matricial Banach Space}\label{constructions}

This section concerns the construction of building matricial Banach spaces from array-weighted sets.  While the main idea is to build the appropriate free algebraic object and construct a universal matrix-norm as in previous works \cite{blackadar1985,gerbracht,goodearl,grilliette1}, the interplay between the levels of an array-weight necessitates a quotient structure in general.  This issue is illustrated in Example \ref{badness5}.

By Example \ref{motivating}, there is a natural forgetful functor $F_{\MBanB}^{\AWSetB}:\MBanB\to\AWSetB$ where all linear structure is removed, leaving the matrix-norm as an array-weight structure.  The goal is now to reverse this process.

\begin{defn2}[Matricial Banach space construction]
Given an array-weighted set $X$, let $V_{X}$ be the free complex vector space on $X$ and $\theta_{X}:X\to V_{X}$ the embedding of generators.  Define
\[
N_{X}:=\left\{v\in V_{X}:\begin{array}{c}\phi(v)=0\forall\phi:V_{X}\to\MIN(\mathbb{C})\textrm{ linear such that }\\\phi\circ\theta_{X}\textrm{ is completely contractive}\end{array}\right\}.
\]
One can check that the set $N_{X}$ is a linear subspace of $V_{X}$.  Let $Q_{X}:=V_{X}/N_{X}$ with quotient map $q_{X}:V_{X}\to Q_{X}$.  Define functions $\|\cdot\|_{Q_{X},m,n}:M_{m,n}\left(Q_{X}\right)\to[-\infty,\infty]$ by
\[
\|A\|_{Q_{X},m,n}:=\sup\left\{\left\|M_{m,n}(\phi)(A)\right\|_{W,m,n}:\begin{array}{l}W\textrm{ a matricial Banach space},\\ \phi:Q_{X}\to W\textrm{ linear},\\ \phi\circ q_{X}\circ\theta_{X}\textrm{ completely contractive}\end{array}\right\}
\]
for all $m,n\in\mathbb{N}$.
\end{defn2}

\begin{lem2}\label{generators-cc}
Equipped with the above functions, $Q_{X}$ is a matrix-normed space and
\[
\left\|M_{m,n}\left(q_{X}\circ\theta_{X}\right)(A)\right\|_{Q_{X},m,n}\leq w_{X,m,n}(A)
\]
for all $m,n\in\mathbb{N}$ and $A\in M_{m,n}(X)$.
\end{lem2}

\begin{proof}

Notice that the supremum nonnegative since the zero map from $Q_{X}$ to $\MIN(\mathbb{C})$ is completely contractive on the generating set $X$.  Now, the inequality on matrices of generators will be shown.  Consider a matricial Banach space $W$ and a linear map $\phi:Q_{X}\to W$ such that $\phi\circ q_{X}\circ\theta_{X}$ is completely contractive.  For $m,n\in\mathbb{N}$ and $A\in M_{m,n}(X)$, observe that
\[\begin{array}{rcl}
\left\|M_{m,n}(\phi)\left(M_{m,n}\left(q_{X}\circ\theta_{X}\right)(A)\right)\right\|_{W,m,n}
&	=	&	\left\|M_{m,n}\left(\phi\circ q_{X}\circ\theta_{X}\right)(A)\right\|_{W,m,n}\\[10pt]
&	\leq	&	w_{X,m,n}(A).\\[10pt]
\end{array}\]
A supremum then yields
\[
\left\|M_{m,n}\left(q_{X}\circ\theta_{X}\right)(A)\right\|_{Q_{X},m,n}\leq w_{X,m,n}(A).
\]

Next, the norm of individual vectors and matrices will be shown to be finite.  Consider a matricial Banach space $W$ and a linear map $\phi:Q_{X}\to W$ such that $\phi\circ q_{X}\circ\theta_{X}$ is completely contractive.  Let $n\in\mathbb{N}$, $\left(x_{j}\right)_{j=1}^{n}\subseteq X$, and $\left(\lambda_{j}\right)_{j=1}^{n}\subset\mathbb{C}$.  Then,
\[\begin{array}{rcl}
\left\|\phi\left(\sum_{j=1}^{n}\lambda_{j}\left(q_{X}\circ \theta_{X}\right)\left(x_{j}\right)\right)\right\|_{W,1,1}
&	=	&	\left\|\sum_{j=1}^{n}\lambda_{j}\left(\phi\circ q_{X}\circ \theta_{X}\right)\left(x_{j}\right)\right\|_{W,1,1}\\[10pt]
&	\leq	&	\sum_{j=1}^{n}\left|\lambda_{j}\right|\left\|\left(\phi\circ q_{X}\circ \theta_{X}\right)\left(x_{j}\right)\right\|_{W,1,1}\\[10pt]
&	\leq	&	\sum_{j=1}^{n}\left|\lambda_{j}\right|w_{X,1,1}\left(x_{j}\right),\\[10pt]
\end{array}\]
so a supremum gives
\[
\left\|\sum_{j=1}^{n}\lambda_{j}\left(q_{X}\circ \theta_{X}\right)\left(x_{j}\right)\right\|_{Q_{X},1,1}\leq\sum_{j=1}^{n}\left|\lambda_{j}\right|w_{X,1,1}\left(x_{j}\right)<\infty.
\]
For $m,n\in\mathbb{N}$, let $B\in M_{m,n}\left(Q_{X}\right)$.  Then,
\[\begin{array}{rcl}
\left\|M_{m,n}(\phi)(B)\right\|_{W,m,n}
&	\leq	&	w_{F_{\MBanB}^{\AWSetB}(W),m,n}\left(M_{m,n}(\phi)(B)\right)\\[10pt]
&	=	&	\sum_{j=1}^{m}\sum_{k=1}^{n}w_{F_{\MBanB}^{\AWSetB}(W),1,1}(\phi(B(j,k)))\\[10pt]
&	=	&	\sum_{j=1}^{m}\sum_{k=1}^{n}\|\phi(B(j,k))\|_{W,1,1}\\[10pt]
&	\leq	&	\sum_{j=1}^{m}\sum_{k=1}^{n}\|B(j,k)\|_{Q_{X},1,1},\\[10pt]
\end{array}\]
and a supremum yields
\[
\left\|B\right\|_{Q_{X},m,n}\leq\sum_{j=1}^{m}\sum_{k=1}^{n}\|B(j,k)\|_{Q_{X},1,1}<\infty.
\]

Say $v\in Q_{X}$ satisfies $\|v\|_{Q_{X},1,1}=0$.  There is $w\in V_{X}$ such that $q_{X}(w)=v$.  For a linear map $\phi:V_{X}\to\MIN(\mathbb{C})$ such that $\phi\circ\theta_{X}$ is completely contractive, there is a unique linear map $\hat{\phi}:Q_{X}\to\MIN(\mathbb{C})$ such that $\hat{\phi}\circ q_{X}=\phi$.  Then, $\hat{\phi}\circ q_{X}\circ\theta_{X}=\phi\circ\theta_{X}$ is completely contractive, meaning
\[
0\leq
|\phi(w)|
=\left|\left(\hat{\phi}\circ q_{X}\right)(w)\right|
=\left|\hat{\phi}(v)\right|
\leq\|v\|_{Q_{X},1,1}
=0.
\]
As $\phi$ was arbitrary, $w\in N_{X}$, and $v=q_{X}(w)=0$.

For $m,n\in\mathbb{N}$ and $B,C\in M_{m,n}\left(Q_{X}\right)$, consider a matricial Banach space $W$ and a linear map $\phi:Q_{X}\to W$ such that $\phi\circ q_{X}\circ\theta_{X}$ is completely contractive.  Then,
\[\begin{array}{rcl}
\left\|M_{m,n}(\phi)(B+C)\right\|_{W,m,n}
&	\leq	&	\left\|M_{m,n}(\phi)(B)\right\|_{W,m,n}+\left\|M_{m,n}(\phi)(C)\right\|_{W,m,n}\\[10pt]
&	\leq	&	\left\|B\right\|_{Q_{X},m,n}+\left\|C\right\|_{Q_{X},m,n},\\[10pt]
\end{array}\]
and a supremum gives
\[
\left\|B+C\right\|_{Q_{X},m,n}\leq\left\|B\right\|_{Q_{X},m,n}+\left\|C\right\|_{Q_{X},m,n}
\]
Let $j,k\in\mathbb{N}$, $D\in\mathbb{M}_{j,m}$, and $E\in\mathbb{M}_{n,k}$.  A quick calculation shows that $M_{j,k}(\phi)(DBE)=D\cdot M_{m,n}(\phi)(B)\cdot E$, so
\[\begin{array}{rcl}
\left\|M_{j,k}(\phi)(DBE)\right\|_{W,j,k}
&	=	&	\left\|D\cdot M_{m,n}(\phi)(B)\cdot E\right\|_{W,j,k}\\[10pt]
&	\leq	&	\|D\|_{\mathbb{M}_{j,m}}\left\|M_{m,n}(\phi)(B)\right\|_{W,m,n}\|E\|_{\mathbb{M}_{n,k}}\\[10pt]
&	\leq	&	\|D\|_{\mathbb{M}_{j,m}}\|B\|_{Q_{X},m,n}\|E\|_{\mathbb{M}_{n,k}}.\\[10pt]
\end{array}\]
A supremum shows that
\[
\|DBE\|_{Q_{X},j,k}\leq\|D\|_{\mathbb{M}_{j,m}}\|B\|_{Q_{X},m,n}\|E\|_{\mathbb{M}_{n,k}}.
\]

\end{proof}

With this matrix-norm constructed, the space is completed to ensure the creation of a matricial Banach space.

\begin{defn2}[Scaled-free matricial Banach space]
Given an array-weighted set $X$, the \emph{scaled-free matricial Banach space} of $X$ is $\MBanSp(X):=\MC\left(Q_{X}\right)$.  Define $\eta_{X}:=\kappa_{Q_{X}}\circ q_{X}\circ\theta_{X}$, the association of generators.  By Lemma \ref{generators-cc}, $\eta_{X}$ is completely contractive.
\end{defn2}

So constructed, the scaled-free matricial Banach space has the following universal property, analogous to \cite[Theorem 3.1.1]{grilliette1}

\begin{thm2}[Universal property of $\MBanSp$]\label{scaled-free-MBan}
Given a matricial Banach space $W$ and a completely bounded map $\phi:X\to F_{\MBanB}^{\AWSetB}(W)$, there is a unique completely bounded linear map $\hat{\phi}:\MBanSp(X)\to W$ such that $F_{\MBanB}^{\AWSetB}\left(\hat{\phi}\right)\circ\eta_{X}=\phi$.  Moreover,
\[
\cbnd(\phi)\geq\left\|\hat{\phi}\right\|_{\CB\left(\MBanSp(X),W\right)}.
\]
\end{thm2}

\begin{proof}

By the universal property of $V_{X}$, there is a unique linear $\varphi:V_{X}\to W$ such that $\varphi\circ\theta_{X}=\phi$.  For $v\in N_{X}$, there is $\psi\in W^{*}$ such that $|\psi(\varphi(v))|=\|\varphi(v)\|_{W,1,1}$ and $\|\psi\|_{W^{*}}=1$.  By Theorem \ref{MIN-prop}, $\psi$ is completely contractive from $W$ to $\MIN(\mathbb{C})$.  Then, $\psi\circ\phi=\psi\circ\varphi\circ\theta_{X}$ is a completely contractive function from $X$ to $\MIN(\mathbb{C})$, so
\[
\|\varphi(v)\|_{W,1,1}
=|\psi(\varphi(v))|
=0.
\]
Thus, $N_{X}\subseteq\ker(\varphi)$, so there is a unique linear $\varpi:Q_{X}\to W$ such that $\varpi\circ q_{X}=\varphi$ by the universal property of the quotient.

If $\cbnd(\phi)=0$, then $\phi(x)=0$ for all $x\in X$, and consequently, $\varphi$ is the zero map, as is $\varpi$.  In this case, $\cbnd(\phi)=0=\left\|\varpi\right\|_{\CB\left(Q_{X},W\right)}$.

If $\cbnd(\phi)\neq 0$, let $\hat{\varpi}:=\frac{1}{\cbnd(\phi)}\varpi$ be the scaled map.  For $m,n\in\mathbb{N}$ and $A\in M_{m,n}(X)$,
\[\begin{array}{rcl}
\left\|M_{m,n}\left(\hat{\varpi}\circ q_{X}\circ\theta_{X}\right)(A)\right\|_{W,m,n}
&	=	&	\frac{1}{\cbnd(\phi)}\left\|M_{m,n}\left(\varpi\circ q_{X}\circ\theta_{X}\right)(A)\right\|_{W,m,n}\\[12pt]
&	=	&	\frac{1}{\cbnd(\phi)}\left\|M_{m,n}(\phi)(A)\right\|_{W,m,n}\\[12pt]
&	\leq	&	w_{X,m,n}(A).\\[12pt]
\end{array}\]
Thus, $\hat{\varpi}\circ q_{X}\circ\theta_{X}$ is completely contractive, so
\[
\frac{1}{\cbnd(\phi)}\left\|M_{m,n}\left(\varpi\right)(B)\right\|_{W,m,n}
=\left\|M_{m,n}\left(\hat{\varpi}\right)(B)\right\|_{W,m,n}
\leq\left\|B\right\|_{Q_{X},m,n}
\]
for all $B\in M_{m,n}\left(Q_{X}\right)$.  Hence, $\varpi$ is completely bounded and
\[
\cbnd(\phi)\geq\left\|\varpi\right\|_{\CB\left(Q_{X},W\right)}.
\]

In either case, there is a unique completely bounded linear map
\[
\hat{\phi}:\MBanSp(X)\to W
\]
such that $\hat{\phi}\circ\kappa_{Q_{X}}=\varpi$ by Theorem \ref{univ-prop-complete} and
\[
\left\|\hat{\phi}\right\|_{\CB(\MBanSp(X),W)}=\left\|\varpi\right\|_{\CB\left(Q_{X},W\right)}\leq\cbnd(\phi).
\]
Notice that
\[
\hat{\phi}\circ\eta_{X}
=\hat{\phi}\circ\kappa_{Q_{X}}\circ q_{X}\circ\theta_{X}
=\varpi\circ q_{X}\circ\theta_{X}
=\varphi\circ\theta_{X}
=\phi
\]
as desired.  Uniqueness follows from universal properties of the free vector space, the quotient vector space, and the completion.

\end{proof}

Thus, $\MBanSp$ is a left adjoint functor, meaning it will behave well with coproducts and other left adjoints.  Applying $\MBanSp$ to the disjoint union array-weighted set from Theorem \ref{awsetb-coproduct} gives the following result.

\begin{cor2}
For an index set $\Lambda$, let $X_{\lambda}$ be an array-weighted set for each $\lambda\in\Lambda$.  Then,
\[
\MBanSp\left({\coprod_{\lambda\in\Lambda}}^{\AWSetC}X_{\lambda}\right)\cong_{\MBanC}{\coprod_{\lambda\in\Lambda}}^{\MBanC}\MBanSp\left(X_{\lambda}\right).
\]
\end{cor2}

Composing forgetful functors, observe that
\[
F^{\WSetB}_{\AWSetB}\circ F^{\AWSetB}_{\MBanB}=F^{\WSetB}_{\BanB}\circ F^{\BanB}_{\MBanB}.
\]
By the composition of left adjoints, both $\MBanSp\circ\MA$ and $\AMAX\circ\BanSp$ qualify as a left adjoint to the forgetful functor composition.  By uniqueness of left adjoints, these two functors must be naturally isomorphic.

\begin{cor2}\label{MA-MBanSP}
Given a weighted set $S$, then
\[\begin{array}{rcl}
\MBanSp(\MA(S))
&	\cong_{\MBanC}	&	\AMAX(\BanSp(S))\\[10pt]
&	\cong_{\MBanC}	&	\AMAX\left(\ell^{1}\left(\left\{s:w_{S}(s)\neq0\right\}\right)\right).\\[10pt]
\end{array}\]
\end{cor2}

On the other hand, the right adjoint $\mA$ trivializes $\MBanSp$ by Example \ref{badness3}.

\begin{prop2}[Failure of $\mA$]
Given a weighted set $S$, then $\MBanSp(\mA(S))$ is the zero space.
\end{prop2}

\begin{proof}

Let $\phi:V_{\mA(S)}\to\MIN(\mathbb{C})$ be a linear map such that $\phi\circ\theta_{\mA(S)}$ is completely contractive.  By Example \ref{badness3}, $\left(\phi\circ\theta_{\mA(S)}\right)(s)=0$ for all $s\in S$.  Hence, $\theta_{\mA(S)}(s)\in N_{\mA(S)}$ for all $s\in S$.  Consequently, $N_{\mA(S)}=V_{\mA(S)}$.

\end{proof}

An immediate question that arises is whether or not the quotient is necessary in the construction of $\MBanSp(X)$.  When all elements of $X$ are array-free, it is pleasantly not.

\begin{prop2}\label{array-free-nullset}
If all $x\in X$ are array-free in $X$, then $N_{X}=\{0\}$.
\end{prop2}

\begin{proof}

Let $v\in N_{X}$.  There are $\left(x_{j}\right)_{j=1}^{n}\subseteq X$ and $\left(\lambda_{j}\right)_{j=1}^{n}\subset\mathbb{C}$ such that $v=\sum_{j=1}^{n}\lambda_{j}\theta_{X}\left(x_{j}\right)$.  For $1\leq j\leq n$, let $\chi_{j}$ be the characteristic function of $x_{j}$.  Then, there is a unique linear $\hat{\chi}_{j}:V_{X}\to\mathbb{C}$ such that $\hat{\chi}_{j}\circ\theta_{X}=\chi_{j}$.  Define $\phi_{j}:=\frac{1}{\cbnd\left(\chi_{j}\right)}\hat{\chi}_{j}$.  For $m,n\in\mathbb{N}$ and $A\in M_{m,n}(X)$,
\[\begin{array}{rcl}
\left\|M_{m,n}\left(\phi_{j}\circ\theta_{X}\right)(A)\right\|_{\MIN(\mathbb{C}),m,n}
&	=	&	\frac{1}{\cbnd\left(\chi_{j}\right)}\left\|M_{m,n}\left(\hat{\chi}_{j}\circ\theta_{X}\right)(A)\right\|_{\MIN(\mathbb{C}),m,n}\\[12pt]
&	=	&	\frac{1}{\cbnd\left(\chi_{j}\right)}\left\|M_{m,n}\left(\chi_{j}\right)(A)\right\|_{\MIN(\mathbb{C}),m,n}\\[12pt]
&	\leq	&	w_{X,m,n}(A).\\[12pt]
\end{array}\]
By definition of $N_{X}$,
\[\begin{array}{rcl}
0
&	=	&	\phi_{j}\left(v\right)\\[12pt]
&	=	&	\frac{1}{\cbnd\left(\chi_{j}\right)}\hat{\chi}_{j}\left(v\right)\\[12pt]
&	=	&	\frac{1}{\cbnd\left(\chi_{j}\right)}\sum_{k=1}^{n}\lambda_{k}\left(\hat{\chi}_{j}\circ\theta_{X}\right)\left(x_{k}\right)\\[12pt]
&	=	&	\frac{1}{\cbnd\left(\chi_{j}\right)}\sum_{k=1}^{n}\lambda_{k}\chi_{j}\left(x_{k}\right)\\[12pt]
&	=	&	\frac{\lambda_{j}}{\cbnd\left(\chi_{j}\right)}.\\[12pt]
\end{array}\]
Hence, $\lambda_{j}=0$.  Since $j$ was arbitrary, $v=0$.

\end{proof}

Combining this with Corollary \ref{finite-array-free} numerically characterizes when $N_{X}$ is trivial for finite array-weighted sets.

\begin{thm2}[Array-free and finite sets, Part II]\label{goodness1}
Given a finite array-weighted set $X$, then $N_{X}=\{0\}$ if and only if $\brn(X)>0$.
\end{thm2}

\begin{proof}

$(\Leftarrow)$ By Corollary \ref{finite-array-free}, all $x\in X$ are array-free in $X$.  By Proposition \ref{array-free-nullset}, $N_{X}=\{0\}$.

$(\Rightarrow)$ As $X$ is finite, $V_{X}$, $Q_{X}$, and $\MBanSp(X)$ are finite-dimensional, so $\kappa_{Q_{X}}$ is an isomorphism of matrix-normed spaces.  Moreover, if $N_{X}=\{0\}$, then $q_{X}$ is a vector space isomorphism as well.

Given any function $\phi:X\to\mathbb{C}$, there is a unique linear map $\varphi:V_{X}\to\mathbb{C}$ such that $\varphi\circ\theta_{X}=\phi$.  There is also a unique linear map $\varpi:Q_{X}\to\mathbb{C}$ such that $\varpi\circ q_{X}=\varphi$.  Since $Q_{X}$ is finite-dimensional, $\varpi$ is bounded and, by Theorem \ref{MIN-prop}, completely bounded from $Q_{X}$ to $\MIN(\mathbb{C})$.  Then, there is a unique completely bounded linear map $\hat{\phi}:\MBanSp(X)\to\MIN(\mathbb{C})$ such that $\hat{\phi}\circ\kappa_{Q_{X}}=\varpi$ by Theorem \ref{univ-prop-complete}.  Notice that $\hat{\phi}\circ\eta_{X}=\phi$ is completely bounded.  Since $\phi$ was arbitrary, Corollary \ref{finite-array-free} states that $\brn(X)>0$.

\end{proof}

In fact, $\MBanSp$ of a singleton array-weighted set is readily computed.

\begin{thm2}[Characterization of singletons, $\MBanSp$]\label{singletons-mbansp}
Let $X=\{x\}$ be a singleton array-weighted set.  Then,
\[
\MBanSp(X)\cong_{\MBanC}\left\{\begin{array}{cc}
\AMAX(\mathbb{C}),	&	\brn(X)>0,\\
\{0\},	&	\brn(X)=0.\\
\end{array}\right.
\]
\end{thm2}

\begin{proof}

Let $\chi:X\to\mathbb{C}$ be the constant map to 1, which is the characteristic function of $x$.  From Corollary \ref{finite-array-free}, $\chi$ is completely bounded to $\MIN(\mathbb{C})$ if and only if $\brn(X)>0$.  Notice that given any map $\phi:X\to\mathbb{C}$, $\phi=\phi(x)\chi$.  Thus, if $\brn(X)=0$, $\phi$ is completely bounded to $\MIN(\mathbb{C})$ if and only if $\phi(x)=0$.  Hence, $\theta_{X}(x)\in N_{X}$, meaning $\MBanSp(X)\cong\{0\}$.

If $\brn(X)>0$, define $\phi:X\to\mathbb{C}$ by $\phi(x):=\brn(X)$.  A quick check shows that $\phi$ is completely contractive from $X$ to $\AMAX(\mathbb{C})$.  Thus, there is a unique completely contractive linear $\hat{\phi}:\MBanSp(X)\to\AMAX(\mathbb{C})$ such that $\hat{\phi}\circ\eta_{X}=\phi$.  Consequently,
\[
\brn(X)\leq\left\|\eta_{X}(x)\right\|_{\MBanSp(X),1,1}.
\]

On the other hand, consider a completely contractive map $\psi:X\to W$ for some arbitrary matricial Banach space $W$.  For $m,n\in\mathbb{N}$ and $A\in M_{m,n}(X)$,
\[
M_{m,n}\left(\psi\right)(A)=\psi(x)\otimes J_{m,n},
\]
where $J_{m,n}$ is the matrix with all entries 1.  Then,
\[\begin{array}{rcl}
\sqrt{mn}\left\|\psi(x)\right\|_{W,1,1}
&	=	&	\left\|\psi(x)\otimes J_{m,n}\right\|_{\MIN\left(F_{\MBanC}^{\BanC}(W)\right),m,n}\\[12pt]
&	\leq	&	\left\|\psi(x)\otimes J_{m,n}\right\|_{W,m,n}\\[12pt]
&	\leq	&	\left\|\psi(x)\otimes J_{m,n}\right\|_{\AMAX\left(F_{\MBanC}^{\BanC}(W)\right),m,n}\\[12pt]
&	=	&	\sqrt{mn}\left\|\psi(x)\right\|_{W,1,1},\\[12pt]
\end{array}\]
forcing equality.  Consequently,
\[
\frac{\left\|M_{m,n}\left(\psi\right)(A)\right\|_{W,m,n}}{w_{X,m,n}(A)}
=\frac{\sqrt{mn}\left\|\psi(x)\right\|_{W,1,1}}{w_{X,m,n}(A)}
\leq 1,
\]
so $\left\|\psi(x)\right\|_{W,1,1}\leq\brn(X)$.  In particular, this states that
\[
\brn(X)=\left\|\eta_{X}(x)\right\|_{\MBanSp(X),1,1}.
\]
Define $\varphi:\mathbb{C}\to\MBanSp(X)$ by $\varphi(\lambda):=\frac{\lambda}{\brn(X)}\eta_{X}(x)$.  This map is immediately linear and contractive by the calculations above.  By Theorem \ref{AMAX-prop}, $\varphi$ is completely contractive from $\AMAX(\mathbb{C})$ to $\MBanSp(X)$.  A routine calculation now shows that $\varphi\circ\hat{\phi}=id_{\MBanSp(X)}$ and $\hat{\phi}\circ\varphi=id_{\AMAX(\mathbb{C})}$.

\end{proof}

Unfortunately, the quotient structure is necessary in general.  The following example shows $N_{X}$ to be nontrivial while not annihilating either generator.  It also yields $\MIN(\mathbb{C})$ rather than $\AMAX(\mathbb{C})$ or $\{0\}$ like the previous examples.

\begin{ex2}\label{badness5}
Let $X$, $\iota$, and $\chi_{x}$ for $x\in X$ be as defined in \ref{badness4}.  Define $v:=\eta_{X}(1)\in\MBanSp(X)$, and $w_{x}:=\theta_{X}(x)\in V_{X}$ for $x\in X$.

First, $\MBanSp(X)$ is characterized as the span of $v$.  Consider a linear map $\phi:V_{X}\to\MIN\left(\mathbb{C}\right)$ such that $\phi\circ\theta_{X}$ is completely contractive.  Observe that
\[
\phi\circ\theta_{X}
=\phi\left(w_{1}\right)\iota
+\left(\phi\left(w_{1}\right)+\phi\left(w_{-1}\right)\right)\chi_{-1}.
\]
Since $\chi_{-1}$ is not completely bounded, $\phi\left(w_{-1}\right)=-\phi\left(w_{1}\right)$, meaning that $\phi\circ\theta_{X}=\phi\left(w_{1}\right)\iota$.  For $\lambda,\mu\in\mathbb{C}$,
\[
\phi\left(\lambda w_{1}+\mu w_{-1}\right)
=(\lambda-\mu)\phi_{1}\left(w_{1}\right),
\]
which is guaranteed to be 0 when $\lambda=\mu$.  Consequently,
\[
N_{X}=\Span\left\{w_{1}+w_{-1}\right\},
\]
and $\MBanSp(X)=\Span\{v\}$.

Next, the norm of $v\otimes I_{m}$ is computed, where $I_{m}$ is the identity of $\mathbb{M}_{m,m}$.  By Theorem \ref{scaled-free-MBan}, there is a unique completely contractive linear map $\hat{\iota}:\MBanSp(X)\to\MIN(\mathbb{C})$ such that $\hat{\iota}\circ\eta_{X}=\iota$.  For $m\in\mathbb{N}$,
\[
1=\left\|I_{m}\right\|_{\mathbb{M}_{m,m}}
=\left\|M_{m,n}\left(\hat{\iota}\right)\left(v\otimes I_{m}\right)\right\|_{\MIN(\mathbb{C}),m,m}
\leq\left\|v\otimes I_{m}\right\|_{\MBanSp(X),m,m}.
\]
Letting $A_{m}$ be defined as in Example \ref{badness4}, then
\[
I_{m}=2^{-m+1}A_{m-1}A_{m-1}^{*}
\]
and
\[
M_{m,2^{m-1}}\left(\eta_{X}\right)\left(A_{m-1}\right)=v\otimes A_{m-1}
\]
for all $m\in\mathbb{W}$.  Therefore,
\[
v\otimes I_{m}
=v\otimes\left(2^{-m+1}A_{m-1}A_{m-1}^{*}\right)
=2^{-m+1}\left(v\otimes A_{m-1}\right)A_{m-1}^{*}
\]
so
\[
\left\|v\otimes I_{m}\right\|_{\MBanSp(X),m,m}
=\left\|2^{-m+1}\left(v\otimes A_{m-1}\right)A_{m-1}^{*}\right\|_{\MBanSp(X),m,m}
\]
\[\begin{array}{cl}
\leq	&	2^{-m+1}\left\|v\otimes A_{m-1}\right\|_{\MBanSp(X),m,2^{m-1}}\left\|A_{m-1}^{*}\right\|_{\mathbb{M}_{2^{m-1},m}}\\[10pt]
=	&	2^{-m+1}\left\|M_{m,2^{m-1}}\left(\eta_{X}\right)\left(A_{m-1}\right)\right\|_{\MBanSp(X),m,2^{m-1}}\left\|A_{m-1}\right\|_{\mathbb{M}_{m,2^{m-1}}}\\[10pt]
\leq	&	2^{-m+1}w_{X,m,2^{m-1}}\left(A_{m-1}\right)\left\|A_{m-1}\right\|_{\mathbb{M}_{m,2^{m-1}}}\\[10pt]
=	&	2^{-m+1}\left\|A_{m-1}\right\|_{\mathbb{M}_{m,2^{m-1}}}^{2}\\[10pt]
=	&	2^{-m+1}\cdot 2^{m-1}\\[10pt]
=	&	1.\\[10pt]
\end{array}\]
Thus, $\left\|v\otimes I_{m}\right\|_{\MBanSp(X),m,m}=1$ for all $m\in\mathbb{N}$.

Lastly, $\hat{\iota}$ is shown to be completely isometric and, thereby, an isomorphism in $\MBanC$.  For all $m,n\in\mathbb{N}$ and $B\in M_{m,n}(\MBanSp(X))$, there is a unique $\hat{B}\in\mathbb{M}_{m,n}$ such that $B=v\otimes\hat{B}$.  Thus,
\[
\left\|\hat{B}\right\|_{\MIN(\mathbb{C}),m,n}
=\left\|M_{m,n}\left(\hat{\iota}\right)\left(B\right)\right\|_{\MIN(\mathbb{C}),m,m}
\leq\left\|B\right\|_{\MBanSp(X),m,m}
\]
and
\[
\|B\|_{\MBanSp(X),m,n}
\leq\left\|v\otimes I_{m}\right\|_{\MBanSp(X),m,m}\left\|\hat{B}\right\|_{\mathbb{M}_{m,n}}
=1\cdot\left\|\hat{B}\right\|_{\MIN(\mathbb{C}),m,n}.
\]
Consequently,
\[
\left\|M_{m,n}\left(\hat{\iota}\right)(B)\right\|_{\MIN(\mathbb{C}),m,n}
=\|B\|_{\MBanSp(X),m,n}
\]
as desired.
\end{ex2}

\section{Matricial Banach Algebras}\label{mbanachalgebra}

This section considers algebras equipped with a matrix-norm compatible with matrix multiplication.  As an algebra is a vector space with a multiplication of vectors, the conventions for algebras used here will be analogous to the vector space conventions used previously.

\begin{defn2}[Matrix conventions, algebras]
For an algebra $\alg{A}$ and $m,n\in\mathbb{N}$, $M_{m,n}(\alg{A})$ is equipped with the same operations from being a vector space:  pointwise addition, pointwise scalar multiplication, and actions of scalar matrices on left and right.  Moreover, the action $M_{m,p}(\alg{A})\times M_{p,n}(\alg{A})\to M_{m,n}(\alg{A})$ will be by matrix multiplication for all $p\in\mathbb{N}$.
\end{defn2}

The notion of an $L^{\infty}$-matrix-normed algebra was introduced in \cite[Definition 1.4]{blecher1990}.  Here, the definition is generalized to consider rectangular matrices without the $L^{\infty}$-condition.

\begin{defn2}
A \emph{matrix-normed algebra} is a complex algebra $\alg{A}$ equipped with a matrix-norm $\left(\|\cdot\|_{\alg{A},m,n}\right)_{m,n\in\mathbb{N}}$ such that
\[
\left\|AB\right\|_{\alg{A},m,n}\leq\left\|A\right\|_{\alg{A},m,p}\left\|B\right\|_{\alg{A},p,n}
\]
for all $m,p,n\in\mathbb{N}$, $A\in M_{m,p}(\alg{A})$, and $B\in M_{p,n}(\alg{A})$.  A complete matrix-normed algebra is a \emph{matricial Banach algebra}.  Let $\MBanAlgB$ be the category of matricial Banach algebras with completely bounded algebra homomorphisms, and $\MBanAlgC$ be the category of matricial Banach algebras with completely contractive algebra homomorphisms.
\end{defn2}

The goal of this section will be to build examples of this structure from various existing structures.  Section \ref{amax-algebra} takes a Banach algebra and imbues it with an extremal matrix-norm.  Section \ref{haageruptensoralgebra} takes a matricial Banach space and creates a matrix-norm on the tensor algebra.  Section \ref{free-product} concludes by proving the existence of a free product of matricial Banach algebras using array-weighted sets.

\subsection{The Absolute Maximum Matricial Banach Algebra}\label{amax-algebra}

An algebra $\alg{A}$ with a matrix-norm is a normed algebra when stripped of all its matrix-norms, except for the norm on $M_{1,1}(\alg{A})\cong\alg{A}$.  For a Banach algebra, one would like to extend its existing norm to a matrix norm.  As with Banach spaces, many such extensions exist, but $\AMAX$ happens to give a matricial Banach algebra structure, as well as being the maximal such structure.

\begin{lem}
Given a Banach algebra $\alg{A}$, $\AMAX(\alg{A})$ is a matricial Banach algebra.
\end{lem}

\begin{proof}

By \cite[Theorem 2.1]{effros1988}, $\AMAX(\alg{A})$ is already a matricial Banach space in addition to being a complex algebra.  All that remains is to prove sub-multiplicativity of the matrix-norm.  To that end, let $m,p,n\in\mathbb{N}$, $A\in M_{m,p}(\alg{A})$, and $B\in M_{p,n}(\alg{A})$.  Write $A=\sum_{l=1}^{j}a_{l}\otimes C_{l}$ and $B=\sum_{k=1}^{q}b_{k}\otimes D_{k}$.  Then,
\[
AB=\sum_{l=1}^{j}\sum_{k=1}^{q}\left(a_{l}b_{k}\right)\otimes\left(C_{l}D_{k}\right),
\]
so
\[\begin{array}{rcl}
\left\|AB\right\|_{\AMAX(\alg{A}),m,n}
&	\leq	&	\sum_{l=1}^{j}\sum_{k=1}^{q}\left\|a_{l}b_{k}\right\|_{\alg{A}}\left\|C_{l}D_{k}\right\|_{\mathbb{M}_{n,m}^{*}}\\[12pt]
&	\leq	&	\sum_{l=1}^{j}\sum_{k=1}^{q}\left\|a_{l}\right\|_{\alg{A}}\left\|b_{k}\right\|_{\alg{A}}\left\|C_{l}\right\|_{\mathbb{M}_{p,m}^{*}}\left\|D_{k}\right\|_{\mathbb{M}_{n,p}^{*}}\\[12pt]
&	=	&	\left(\sum_{l=1}^{j}\left\|a_{l}\right\|_{\alg{A}}\left\|C_{l}\right\|_{\mathbb{M}_{p,m}^{*}}\right)\left(\sum_{k=1}^{q}\left\|b_{k}\right\|_{\alg{A}}\left\|D_{k}\right\|_{\mathbb{M}_{n,p}^{*}}\right).\\[12pt]
\end{array}\]
Two infima then yield
\[
\left\|AB\right\|_{\AMAX(\alg{A}),m,n}\leq\left\|A\right\|_{\AMAX(\alg{A}),m,p}\left\|B\right\|_{\AMAX(\alg{A}),p,n}.
\]

\end{proof}

Thus, an adaptation of Theorem \ref{AMAX-prop} shows that $\AMAX$ serves as left adjoint to a second forgetful functor.

\begin{thm}[Universal property of $\AMAX$, algebra version]
Let $\BanAlgB$ denote the category of Banach algebras with bounded algebra homomorphisms, and $F_{\MBanAlgB}^{\BanAlgB}:\MBanAlgB\to\BanAlgB$ be the forgetful functor stripping all matrix-norm structure except the underlying norm.  For a Banach algebra $\alg{A}$ and a matricial Banach algebra $\alg{B}$, consider a bounded algebra homomorphism $\phi:\alg{A}\to F_{\MBanAlgB}^{\BanAlgB}(\alg{B})$.  Then, there is a unique completely bounded algebra homomorphism $\hat{\phi}:\AMAX(\alg{A})\to\alg{B}$ such that $F_{\MBanAlgB}^{\BanAlgB}\left(\hat{\phi}\right)=\phi$.  Moreover,
\[
\left\|\hat{\phi}\right\|_{\CB(\AMAX(\alg{A}),\alg{B})}=\|\phi\|_{\B\left(\alg{A},F_{\MBanAlgB}^{\BanAlgB}(\alg{B})\right)}.
\]
\end{thm}

Again, the proof of the above theorem is nearly identical to \cite[Exercise 14.1]{paulsen}.

\subsection{Haagerup Tensor Algebra}\label{haageruptensoralgebra}

A matrix-normed algebra is a matrix-normed space when stripped of its multiplicative structure.  For a matricial Banach space, one would like to construct a multiplicative structure much like the Banach tensor algebra.  Indeed, this can be accomplished by merging the matricial $\ell^{1}$-direct sum and Haagerup tensor product.

\begin{defn}
Given a matricial Banach space $V$, inductively define the \emph{Haagerup tensor powers} of $V$ in the following way:
\[\begin{array}{ccc}
V^{\otimes_{h}1}:=V,	&	V^{\otimes_{h}(n+1)}:=\left(V^{\otimes_{h}n}\right)\otimes_{h}V & \forall n\in\mathbb{N}.
\end{array}\]
The \emph{Haagerup tensor algebra} of $V$ is
\[
\Th(V):={\coprod_{n\in\mathbb{N}}}^{\MBanC}V^{\otimes_{h}n},
\]
the matricial $\ell^{1}$-direct sum of these Haagerup tensor powers, equipped with the multiplication is determined by the canonical isomorphisms $\left(V^{\otimes_{h}m}\right)\otimes_{h}\left(V^{\otimes_{h}n}\right)\to V^{\otimes_{h}(m+n)}$.  Define $\epsilon_{V}:V\to\Th(V)$ to be the inclusion map into the first tensor power of $V$ in $\Th(V)$.
\end{defn}

\begin{lem}
Given a matricial Banach space $V$, $\Th(V)$ is a matricial Banach algebra.
\end{lem}

\begin{proof}

By construction, $\Th(V)$ is a matricial Banach space, and one can check that the tensor multiplication makes $\Th(V)$ into a complex algebra.  All that remains to show is that the matrix-norm is sub-multiplicative on matrices.

To that end, let $m,p,n\in\mathbb{N}$, $A\in M_{m,p}\left(\Th(V)\right)$, and $B\in M_{p,n}\left(\Th(V)\right)$.  Write $A(j,l)=\sum_{q=1}^{\infty}x_{j,l,q}$ and $B(l,k)=\sum_{r=1}^{\infty}y_{l,k,r}$, where $x_{j,l,q},y_{l,k,q}\in V^{\otimes_{h}q}$ for all $1\leq j\leq m$, $1\leq l\leq p$, $1\leq k\leq n$, and $q\in\mathbb{N}$.  Then,
\[\begin{array}{rcl}
(AB)(j,k)
&	=	&	\sum_{l=1}^{p}A(j,l)B(l,k)\\[20pt]
&	=	&	\sum_{q=1}^{\infty}\sum_{r=1}^{\infty}\sum_{l=1}^{p}x_{j,l,q}\otimes y_{l,k,r}\\[20pt]
&	=	&	\sum_{q=1}^{\infty}\sum_{r=1}^{\infty}\left(M_{m,p}\left(\pi_{q}\right)(A)\odot M_{p,n}\left(\pi_{r}\right)(B)\right)(j,k),\\[20pt]
\end{array}\]
where $\pi_{q}$ and $\pi_{r}$ are the coordinate projections of the $\ell^{1}$-sum.  Then,
\[\begin{array}{rcl}
\left\|AB\right\|_{\Th(V),m,n}
&	\leq	&	\sum_{q=1}^{\infty}\sum_{r=1}^{\infty}\left\|M_{m,p}\left(\pi_{q}\right)(A)\odot M_{p,n}\left(\pi_{r}\right)(B)\right\|_{V^{\otimes_{h}(q+r)},m,n}\\[20pt]
\end{array}\]
\[\begin{array}{rcl}
&	\leq	&	\sum_{q=1}^{\infty}\sum_{r=1}^{\infty}\left\|M_{m,p}\left(\pi_{q}\right)(A)\right\|_{V^{\otimes_{h}q},m,p}\left\|M_{p,n}\left(\pi_{r}\right)(B)\right\|_{V^{\otimes_{h}r},p,n}\\[20pt]
&	=	&	\left(\sum_{q=1}^{\infty}\left\|M_{m,p}\left(\pi_{q}\right)(A)\right\|_{V^{\otimes_{h}q},m,p}\right)\left(\sum_{r=1}^{\infty}\left\|M_{p,n}\left(\pi_{r}\right)(B)\right\|_{V^{\otimes_{h}r},p,n}\right)\\[20pt]
&	=	&	\left\|A\right\|_{\Th(V),m,p}\left\|B\right\|_{\Th(V),p,n}.\\[20pt]
\end{array}\]

\end{proof}

So constructed, $\Th(V)$ enjoys the following universal property, analogous to \cite[Satz 1]{leptin1969}.

\begin{thm}[Universal property of the Haagerup tensor algebra]\label{universal-prop-haagerup}
Let
\[
F_{\MBanAlgC}^{\MBanC}:\MBanAlgC\to\MBanC
\]
be the forgetful functor stripping multiplicative structure.  For a matricial Banach space $V$ and a matricial Banach algebra $\alg{B}$, consider a completely contractive linear map $\phi:V\to F_{\MBanAlgC}^{\MBanC}(\alg{B})$.  Then, there is a unique completely contractive algebra homomorphism $\hat{\phi}:\Th(V)\to\alg{B}$ such that $F_{\MBanAlgC}^{\MBanC}\left(\hat{\phi}\right)\circ\epsilon_{V}=\phi$.
\end{thm}

\begin{proof}

First, completely contractive maps are inductively constructed on the Haagerup tensor powers of $V$.  Let $\phi_{1}:=\phi$ be regarded as a completely contractive linear map from $V^{\otimes_{h}1}$ to $\alg{B}$.  For induction, assume for some $n\in\mathbb{N}$ that there are completely contractive linear maps $\phi_{j}:V^{\otimes_{h}j}\to\alg{B}$ for all $1\leq j\leq n$.  Define $\varphi:V^{\otimes_{h}n}\times V\to\alg{B}$ by $\varphi(w,v):=\phi_{n}(w)\cdot\phi(v)$, which is readily seen to be bilinear.  Consider $m,p,q\in\mathbb{N}$, $A\in M_{m,p}\left(V^{\otimes_{h}n}\right)$ and $B\in M_{p,q}(V)$.  Then,
\[\begin{array}{rcl}
(A\odot_{\varphi}B)(i,k)
&	=	&	\sum_{l=1}^{p}\varphi\left(A(i,l),B(l,k)\right)\\[15pt]
&	=	&	\sum_{l=1}^{p}\phi_{n}\left(A(i,l)\right)\cdot\phi\left(B(l,k)\right)\\[15pt]
&	=	&	\left(M_{m,p}\left(\phi_{n}\right)(A)\cdot M_{p,q}(\phi)(B)\right)(i,k)\\[15pt]
\end{array}\]
for all $1\leq i\leq m$ and $1\leq k\leq q$.  Hence,
\[\begin{array}{rcl}
\left\|A\odot_{\varphi}B\right\|_{\alg{B},m,q}
&	=	&	\left\|M_{m,p}\left(\phi_{n}\right)(A)\cdot M_{p,q}(\phi)(B)\right\|_{\alg{B},m,q}\\[15pt]
&	\leq	&	\left\|M_{m,p}\left(\phi_{n}\right)(A)\right\|_{\alg{B},m,p}\left\|M_{p,q}(\phi)(B)\right\|_{\alg{B},p,q}\\[15pt]
&	\leq	&	\left\|A\right\|_{V^{\otimes_{h}n},m,p}\left\|B\right\|_{V,p,q}.\\[15pt]
\end{array}\]
Thus, $\varphi$ is completely contractive bilinear.  By Theorem \ref{haagerup}, there is a unique completely contractive linear map $\phi_{n+1}:V^{\otimes(n+1)}\to\alg{B}$ such that $\phi_{n+1}(w\otimes v)=\phi_{n}(w)\cdot\phi(v)$.

Next, the maps are combined using the matricial $\ell^{1}$-direct sum.  By Theorem \ref{MBan-coproduct}, there is a unique completely contractive linear map $\hat{\phi}:\Th(V)\to\alg{B}$ such that $\hat{\phi}\circ\varpi_{n}=\phi_{n}$ for all $n\in\mathbb{N}$.  In particular, notice that $\epsilon_{V}=\varpi_{1}$, so $\hat{\phi}\circ\epsilon_{V}=\phi_{1}=\phi$ as desired.  Uniqueness of $\hat{\phi}$ arises from the universal properties of $\otimes_{h}$ and ${\coprod}^{\MBanC}$.

All that remains is to show that $\hat{\phi}$ is multiplicative.  Given $v,w\in V$, note that
\[
\hat{\phi}(v\otimes w)
=\phi_{2}(v\otimes w)
=\phi(v)\cdot\phi(w)
=\hat{\phi}(v)\cdot\hat{\phi}(w)
\]
by construction of $\hat{\phi}$ and $\phi_{2}$.  For induction, assume that for some $n\in\mathbb{N}$,
\[
\hat{\phi}\left(v_{1}\otimes\cdots\otimes v_{n}\right)=\hat{\phi}\left(v_{1}\right)\cdots\hat{\phi}\left(v_{n}\right).
\]
for all $\left(v_{j}\right)_{j=1}^{n}\subset V$.  For any $w\in V$,
\[\begin{array}{rcl}
\hat{\phi}\left(v_{1}\otimes\cdots\otimes v_{n}\otimes w\right)
&	=	&	\phi_{n+1}\left(v_{1}\otimes\cdots\otimes v_{n}\otimes w\right)\\[10pt]
&	=	&	\phi_{n}\left(v_{1}\otimes\cdots\otimes v_{n}\right)\cdot\phi\left(w\right)\\[10pt]
&	=	&	\hat{\phi}\left(v_{1}\right)\cdots\hat{\phi}\left(v_{n}\right)\cdot\phi\left(w\right)\\[10pt]
&	=	&	\hat{\phi}\left(v_{1}\right)\cdots\hat{\phi}
\left(v_{n}\right)\cdot\hat{\phi}\left(w\right)\\[10pt]
\end{array}\]
by construction of $\hat{\phi}$ and $\phi_{n+1}$.  Linearity and continuity of $\hat{\phi}$ extend this multiplicativity from elementary tensors to all elements of $\Th(V)$.

\end{proof}

As a left adjoint functor, $\Th$ behaves well with coproducts and other left adjoints.  Let $F^{\BanC}_{\MBanC}$ and $F_{\MBanAlgC}^{\BanAlgC}$ be the completely contractive variations of the forgetful functors stripping all norms, except the underlying norm.  Composing forgetful functors, observe that
\[
F^{\BanC}_{\MBanC}\circ F_{\MBanAlgC}^{\MBanC}=F^{\BanC}_{\BanAlgC}\circ F_{\MBanAlgC}^{\BanAlgC}.
\]
By the composition of left adjoints, both $\Th\circ\AMAX$ and $\AMAX\circ\mathcal{T}$ qualify as a left adjoint to the forgetful functor composition.  By uniqueness of left adjoints, these two functors must be naturally isomorphic.

\begin{cor}\label{th-amax}
Given a Banach space $V$,
\[
\Th(\AMAX(V))
\cong_{\MBanAlgC}
\AMAX(\mathcal{T}(V)).
\]
\end{cor}

Composing left adjoints $\MBanSp$ and $\Th$ creates a new left adjoint, $\MBanAlg:=\Th\circ\MBanSp$, with the following universal property.

\begin{thm}[Universal property of $\MBanAlg$]\label{mbanalg-prop}
Let
\[
F_{\MBanAlgC}^{\AWSetC}:\MBanAlgC\to\AWSetC
\]
be the forgetful functor stripping all algebraic structure.  For an array-weighted set $X$ and a matricial Banach algebra $\alg{B}$, consider a completely contractive map $\phi:X\to F_{\MBanAlgC}^{\AWSetC}(\alg{B})$.  Then, there is a unique completely contractive algebra homomorphism $\hat{\phi}:\MBanAlg(X)\to\alg{B}$ such that $F_{\MBanAlgC}^{\AWSetC}\left(\hat{\phi}\right)\circ\epsilon_{\MBanSp(X)}\circ\eta_{X}=\phi$.
\end{thm}

\begin{proof}

Letting $F_{\MBanC}^{\AWSetC}$ be the completely contractive version of the forgetful functor stripping all vector space structure, notice that
\[
F_{\MBanC}^{\AWSetC}\circ F_{\MBanAlgC}^{\MBanC}=F_{\MBanAlgC}^{\AWSetC}.
\]
Thus, $\phi:X\to F_{\MBanC}^{\AWSetC}\left(F_{\MBanAlgC}^{\MBanC}(\alg{B})\right)$, so there is a unique completely contractive linear map $\tilde{\phi}:\MBanSp(X)\to F_{\MBanAlgC}^{\MBanC}(\alg{B})$ such that $F_{\MBanC}^{\AWSetC}\left(\hat{\phi}\right)\circ\eta_{X}=\phi$ by Theorem \ref{scaled-free-MBan}.  By Theorem \ref{universal-prop-haagerup}, there is a unique completely contractive algebra homomorphism $\hat{\phi}:\Th(\MBanSp(X))\to\alg{B}$ such that $F_{\MBanAlgC}^{\MBanC}\left(\hat{\phi}\right)\circ\epsilon_{\MBanSp(X)}=\tilde{\phi}$.

Consequently, $\hat{\phi}:\MBanAlg(X)\to\alg{B}$ and
\[
F_{\MBanAlgC}^{\AWSetC}\left(\hat{\phi}\right)\circ\epsilon_{\MBanSp(X)}\circ\eta_{X}
=\tilde{\phi}\circ\eta_{X}
=\phi.
\]
Uniqueness follows from the universal properties of $\Th$ and $\MBanSp$.

\end{proof}

Combining Theorem \ref{singletons-mbansp} with Corollary \ref{th-amax} gives the following characterization for singleton array-weighted sets.

\begin{cor}[Characterization of singletons, $\MBanAlg$]
Let $X=\{x\}$ be a singleton array-weighted set.  Then,
\[
\MBanAlg(X)\cong_{\MBanAlgC}\left\{\begin{array}{cc}
\AMAX\left(\ell^{1}\right),	&	\brn(X)>0,\\
\{0\},	&	\brn(X)=0.\\
\end{array}\right.
\]
\end{cor}

\begin{proof}

In the case $\brn(X)=0$, then direct calculation from Theorem \ref{singletons-mbansp} gives
\[
\MBanAlg(X)
\cong_{\MBanAlgC}\Th(\{0\})
\cong_{\MBanAlgC}\{0\}.
\]
When $\brn(X)>0$, then Theorem \ref{singletons-mbansp} and Corollary \ref{th-amax} give
\[\begin{array}{rcl}
\MBanAlg(X)
&	\cong_{\MBanAlgC}	&	\Th(\AMAX(\mathbb{C}))\\
&	\cong_{\MBanAlgC}	&	\AMAX(\mathcal{T}(\mathbb{C}))\\
&	\cong_{\MBanAlgC}	&	\AMAX(\ell^{1}),\\
\end{array}\]
where $\ell^{1}$ has the convolution product.

\end{proof}

Using Corollary \ref{MA-MBanSP}, the matricial Banach algebra of an array-weighted set with the maximum structure can be computed also.  Note that the coproduct of Banach algebras is the free product from \cite[Definition 1.4]{gronbaek1992}

\begin{cor}
Given a weighted set $S$,
\[
\MBanAlg(\MA(S))\cong_{\MBanAlgC}\AMAX\left({\coprod_{w_{S}(s)\neq 0}}^{\BanAlgC}\ell^{1}\right).
\]
\end{cor}

\begin{proof}

From definition,
\[
\MBanAlg(\MA(S))
=\Th(\MBanSp(\MA(S))),
\]
so by Corollary \ref{MA-MBanSP}
\[
\MBanAlg(\MA(S))
\cong_{\MBanAlgC}\Th(\AMAX(\BanSp(S))).
\]
By Corollary \ref{th-amax},
\[
\MBanAlg(\MA(S))
\cong_{\MBanAlgC}\AMAX(\mathcal{T}(\BanSp(S))).
\]
By Proposition \ref{decomp-wset},
\[
\MBanAlg(\MA(S))
\cong_{\MBanAlgC}\AMAX\left(\mathcal{T}\left(\BanSp\left({\coprod_{s\in S}}^{\WSetC}W_{w_{S}(s)}(\{s\})\right)\right)\right).
\]
As $\BanSp$ is a left adjoint, the coproduct can be brought out.
\[
\MBanAlg(\MA(S))
\cong_{\MBanAlgC}\AMAX\left(\mathcal{T}\left({\coprod_{s\in S}}^{\BanC}\BanSp\left(W_{w_{S}(s)}(\{s\})\right)\right)\right)
\]
From direct calculation, $\BanSp\left(W_{w_{S}(s)}(\{s\})\right)\cong_{\BanC}\mathbb{C}$ if $w_{S}(s)\neq0$, and the zero space otherwise.  As such, the zero cofactors can be ignored.
\[
\MBanAlg(\MA(S))
\cong_{\MBanAlgC}\AMAX\left(\mathcal{T}\left({\coprod_{w_{S}(s)\neq 0}}^{\BanC}\mathbb{C}\right)\right)
\]
As $\mathcal{T}$ is a left adjoint, the coproduct can be brought out once again.
\[
\MBanAlg(\MA(S))
\cong_{\MBanAlgC}\AMAX\left({\coprod_{w_{S}(s)\neq 0}}^{\BanAlgC}\mathcal{T}\left(\mathbb{C}\right)\right)
\]
Finally, direct computation shows that $\mathcal{T}(\mathbb{C})\cong_{\BanAlgC}\ell^{1}$ equipped with the convolution product.
\[
\MBanAlg(\MA(S))
\cong_{\MBanAlgC}\AMAX\left({\coprod_{w_{S}(s)\neq 0}}^{\BanAlgC}\ell^{1}\right)
\]

\end{proof}

Lastly, observe that all the statements for $\Th$ have been made with completely contractive maps, rather than completely bounded maps.  While one would like to construct a tensor algebra compatible with completely bounded maps, this cannot be done.  The reason is that the multiplication of generators would become unbounded, as demonstrated in the proposition below.

\begin{prop}
Let $F_{\MBanAlgB}^{\AWSetB}$ be the forgetful functor removing all algebraic structure.  An array weighted set $X$ has a reflection along $F_{\MBanAlgB}^{\AWSetB}$ if and only if the only completely bounded map from $X$ to $\MIN(\mathbb{C})$ is the zero map.  In this case, the reflection is the zero algebra equipped with the constant map from $X$.
\end{prop}

The proof of this proposition is nearly identical to \cite[Proposition 3.2.1]{grilliette1}.  Moreover, this proposition and Theorem \ref{scaled-free-MBan} give the following nonexistence result, analogous to \cite[Corollary 3.2.5]{grilliette1}

\begin{cor}
There cannot exist a functor that is left adjoint to the forgetful functor from $\MBanAlgB$ to $\MBanB$, which strips all multiplicative structure.
\end{cor}

\subsection{Free Product Matricial Banach Algebra}\label{free-product}

This section constructs the coproduct of matricial Banach algebras, the free product matricial Banach algebra.  This is directly parallel to the free product of Banach algebras \cite[Definition 1.4]{gronbaek1992}, operator algebras \cite[Theorem 4.1]{blecher1991-2}, and C*-algebras \cite{avitzour1982}.  However, with the scaled-free matricial Banach algebra from Theorem \ref{mbanalg-prop}, construction of this object will be much more algebraic like \cite[\S3.2]{grilliette2}.

\begin{defn}[Free product matricial Banach algebra]
Let $\left(\alg{A}_{\lambda}\right)_{\lambda\in\Lambda}$ be matricial Banach algebras.  Define
\[
G_{\lambda}:=F_{\MBanAlgC}^{\AWSetC}\left(\alg{A}_{\lambda}\right),
\]
the underlying array-weighted set of each $\alg{A}_{\lambda}$, and $G:={\coprod_{\lambda\in\Lambda}}^{\AWSetC}G_{\lambda}$ their coproduct array-weighted set with inclusion maps $\rho_{\lambda}:G_{\lambda}\to G$.  Define $\alg{B}:=\MBanAlg(G)$, the scaled-free matricial Banach algebra of $G$ with map of generators $\upsilon_{G}:G\to F_{\MBanAlgC}^{\AWSetC}\left(\alg{B}\right)$ by $\upsilon_{G}:=\epsilon_{\MBanSp(G)}\circ\eta_{G}$.  Let $J$ be the closed ideal in $\alg{B}$ generated by
\[
\bigcup_{\lambda\in\Lambda}\left\{\begin{array}{l}
\upsilon_{G}\left(\rho_{\lambda}(a+b)\right)-\left(\upsilon_{G}\left(\rho_{\lambda}(a)\right)+\upsilon_{G}\left(\rho_{\lambda}(b)\right)\right),\\
\upsilon_{G}\left(\rho_{\lambda}(ab)\right)-\upsilon_{G}\left(\rho_{\lambda}(a)\right)\upsilon_{G}\left(\rho_{\lambda}(b)\right),\\
\upsilon_{G}\left(\rho_{\lambda}(\mu a)\right)-\mu\upsilon_{G}\left(\rho_{\lambda}(a)\right)\\
\end{array}
:a,b\in\alg{A}_{\lambda},\mu\in\mathbb{C}\right\}
\]
and $\alg{A}:=\alg{B}/J$ the quotient matricial Banach algebra with quotient map $q:\alg{B}\to\alg{A}$.  Defining $\psi_{\lambda}:=q\circ\upsilon_{G}\circ\rho_{\lambda}$ for all $\lambda\in\Lambda$, each $\psi_{\lambda}$ is a completely contractive algebra homomorphism by construction of $G$ and $J$.
\end{defn}

So constructed, $\alg{A}$ has the following universal property.

\begin{thm}[Universal property of the coproduct, $\MBanAlgC$]
For a matricial Banach algebra $\alg{C}$, let $\phi_{\lambda}:\alg{A}_{\lambda}\to\alg{C}$ be a completely contractive algebra homomorphism for all $\lambda\in\Lambda$.  Then, there is a unique completely contractive homomorphism $\phi:\alg{A}\to\alg{C}$ such that $\phi\circ\psi_{\lambda}=\phi_{\lambda}$ for all $\lambda\in\Lambda$.
\end{thm}

\begin{proof}

By Theorem \ref{awsetb-coproduct}, there is a unique completely contractive function $\varphi:G\to F_{\MBanAlgC}^{\AWSetC}\left(\alg{C}\right)$ such that $\varphi\circ\rho_{\lambda}=F_{\MBanAlgC}^{\AWSetC}\left(\phi_{\lambda}\right)$ for all $\lambda\in\Lambda$.  By Theorem \ref{mbanalg-prop}, there is a unique completely contractive homomorphism $\varpi:\alg{B}\to\alg{C}$ such that $F_{\MBanAlgC}^{\AWSetC}\left(\varpi\right)\circ\upsilon_{G}=\varphi$.  Routine calculations show that $J\subseteq\ker(\varpi)$, so there is a unique completely contractive homomorphism $\phi:\alg{A}\to\alg{C}$ such that $\phi\circ q=\varpi$.  Thus,
\[
\phi\circ\psi_{\lambda}
=\phi\circ q\circ\upsilon_{G}\circ\rho_{\lambda}
=\varpi\circ\upsilon_{G}\circ\rho_{\lambda}
=\varphi\circ\rho_{\lambda}
=\phi_{\lambda}
\]
as desired.  Uniqueness follows from the universal properties of $G$, $\alg{B}$, and the quotient.

\end{proof}

Moreover, use of the universal property shows that the homomorphisms $\psi_{\lambda}$ are completely isometric.

\bibliographystyle{amsplain}

\end{document}